\documentclass{amsart}
\usepackage{amssymb,amsfonts, amsthm, amscd, amsmath, mathcomp}
\usepackage{graphicx,float}
\usepackage{ctable}
\usepackage{booktabs}
\usepackage{subfig}
\usepackage{caption}
\usepackage{tikz}

\newtheorem{defn0}{Definition}[section]
\newtheorem{prop0}[defn0]{Proposition}
\newtheorem{conj0}[defn0]{Conjecture}
\newtheorem{thm0}[defn0]{Theorem}
\newtheorem{lem0}[defn0]{Lemma}
\newtheorem{corollary0}[defn0]{Corollary}
\newtheorem{example0}[defn0]{Example}
\newtheorem{qstn0}[defn0]{Question}
\newcommand\res[1]{{\lower1pt\hbox{$|$}}_{\raise.5pt\hbox{${\scriptstyle #1}$}}}

\newenvironment{defn}{\begin{defn0}}{\end{defn0}}
\newenvironment{prop}{\begin{prop0}}{\end{prop0}}
\newenvironment{conj}{\begin{conj0}}{\end{conj0}}
\newenvironment{thm}{\begin{thm0}}{\end{thm0}}
\newenvironment{lem}{\begin{lem0}}{\end{lem0}}
\newenvironment{cor}{\begin{corollary0}}{\end{corollary0}}
\newenvironment{exm}{\begin{example0}\rm}{\end{example0}}

\usepackage{hyperref}

\begin{document}

\newcommand{\QC}{\mathcal{Q}}
\newcommand{\PC}{\mathcal{P}}
\newcommand{\OC}{\mathcal{O}}
\newcommand{\R}{\mathbb{R}}
\newcommand{\Z}{\mathbb{Z}}
\newcommand{\A}{\mathcal{A}}

\title{Lattice-Supported Splines on Polytopal Complexes}

\author{Michael DiPasquale} \thanks{2010\textit{Mathematics Subject Classification.} Primary 13C05, Secondary 13P25, 13C10.
Author supported by National Science Foundation grant DMS 0838434 ``EMSW21MCTP: Research Experience for Graduate Students.''}
\address{Department of Mathematics, University of Illinois, Urbana, Illinois, 61801}
\email{dipasqu1@illinois.edu}

\begin{abstract}
We study the module $C^r(\PC)$ of piecewise polynomial functions of smoothness
$r$ on a pure $n$-dimensional polytopal complex $\PC\subset\R^n$, via an 
analysis of certain subcomplexes $\PC_W$ obtained from the intersection lattice of the interior codimension one faces of $\PC$. We obtain two main results: first, we show that the vector space $C^r_d(\PC)$ of splines of degree $\leq d$ has a basis consisting of splines supported on the $\PC_W$ for $d\gg0$.  We call such splines \textit{lattice-supported}. This shows that an analog of the notion of a star-supported basis for $C^r_d(\Delta)$ studied by Alfeld-Schumaker in the simplicial case holds \cite{NonEx}. Second, we provide a pair of conjectures, one involving lattice-supported splines, bounding how large $d$ must be so that $\mbox{dim}_\R C^r_d(\PC)$ agrees with the McDonald-Schenck formula \cite{TSchenck08}.  A family of examples shows that the latter conjecture is tight.  The proposed bounds generalize known and conjectured bounds in the simplicial case.
\end{abstract}

\maketitle

\section{Introduction}
Let $\PC$ be a subdivision of a region in $\R^n$ by convex polytopes.  $C^r(\PC)$ denotes the set of piecewise polynomial functions (splines) on $\PC$ that are continuously differentiable of order $r$. Study of the spaces $C^r(\PC)$ is a fundamental topic in approximation theory and numerical analysis (see \cite{Boor}) while within the past decade geometric connections have been made between $C^0(\PC)$ and equivariant cohomology rings of toric varieties \cite{Paynes}.  Practical applications of splines include computer aided design, surface modeling, and computer graphics \cite{Boor}.

A central problem in spline theory is to determine the dimension of (and a basis for) the vector space $C^r_d(\PC)$ of splines whose restriction to each facet of $\PC$ has degree at most $d$. In the bivariate, simplicial case, these questions are studied by Alfeld and Schumaker in \cite{AS4r} and \cite{AS3r} using Bernstein-Bezier methods. A signature result appears in \cite{AS3r}, which gives a dimension formula for $C^r_d(\PC)$ when $d \ge 3r+1$ and $\PC$ is a generic simplicial complex.  An algebraic approach to the dimension question was pioneered by Billera in \cite{Homology} using homological and commutative algebra.  This method has been refined and extended by Schenck, Stillman, and McDonald (\cite{LCoho} and \cite{TSchenck08}). The last of these gives a polyhedral version of the Alfeld-Schumaker formula in the planar case, building on work of Rose \cite{r1}, \cite{r2} on dual graphs.

In applications it is often important to find a basis of $C^r_d(\PC)$ which is ``locally supported'' in some sense.  In the simplicial case the natural thing to require is that the basis elements are supported on stars of vertices. Alfeld and Schumaker \cite{NonEx} call such a basis \textit{minimally supported} or \textit{star-supported}, and they show that for $d=3r+1$ it is not always possible to construct a star-supported basis of $C^r_d(\PC)$.  However in the planar simplicial case the bases constructed for $C^r_d(\Delta)$ in \cite{HongDong} and \cite{SuperSpline} for $d\geq 3r+2$ are in fact star-supported.  Alfeld, Schumaker, and Sirvent show in \cite{LocSup} that in the trivariate case $C^r_d(\Delta)$ has a star-supported basis for $d>8r$.

In this paper, we first show that there are polyhedral analogs of ``star-supported splines''. Our main technical tool is Proposition~\ref{local}, which utilizes polytopal subcomplexes $\PC_W\subset\PC$ associated to certain linear subspaces $W\subset\R^n$ to give a precise description of localization of the module $C^r(\PC)$.  We use this description to tie together local characterizations of projective dimension and freeness due to Yuzvinsky \cite{Yuz} and Billera and Rose \cite{Modules}. A consequence of Proposition~\ref{local} is that there are generators for $C^r(\PC)$ as an $R=\R[x_1,\cdots,x_n]$-module which are supported on the complexes $\PC_W$.  From this follows Theorem~\ref{main}, that for $d \gg 0$, $C^r_d(\PC)$ has an $\R$-basis which is supported on subcomplexes of the form $\PC_W$.  We call such a basis a \textit{lattice}-supported basis, where the lattice (in the sense of a graded poset) of interest is the intersection lattice of the interior codimension one faces of $\PC$.  A lattice-supported basis reduces to a star-supported basis in the simplicial case.

In $\S 5$ we use the \textit{regularity} of a graded module to analyze $\mbox{dim}_\R C^r_d(\PC)$, which is also done in \cite{CohVan}.  We propose in Conjecture~\ref{c1} a regularity bound on the module of locally supported splines in the case $\PC\subset\R^2$.  If true this conjecture gives a bound for when the McDonald-Schenck formula \cite{TSchenck08} for $\mbox{dim}_\R C^r_d(\PC)$ holds which generalizes the Alfeld-Schumaker $3r+1$ bound in the planar simplicial case.  We also propose (Conjecture~\ref{c2}) a stronger regularity bound on $C^r(\widehat{\PC})$ for $\PC\subset\R^2$ and give a family of examples to show that this proposed bound is tight.  Conjecture~\ref{c2} generalizes a conjecture of Schenck that the Alfeld-Schumaker dimension formula \cite{AS3r} holds when $d\ge 2r+1$ in the planar simplicial case \cite{Thesis}.

\subsection{Example of locally supported splines}\label{EX1}
Consider the two dimensional polytopal complex $\QC$ in Figure~\ref{SEVL} with 5 faces, 8 interior edges, and 4 interior vertices.
\begin{minipage}[h]{.4 \textwidth}
\includegraphics[scale=.5]{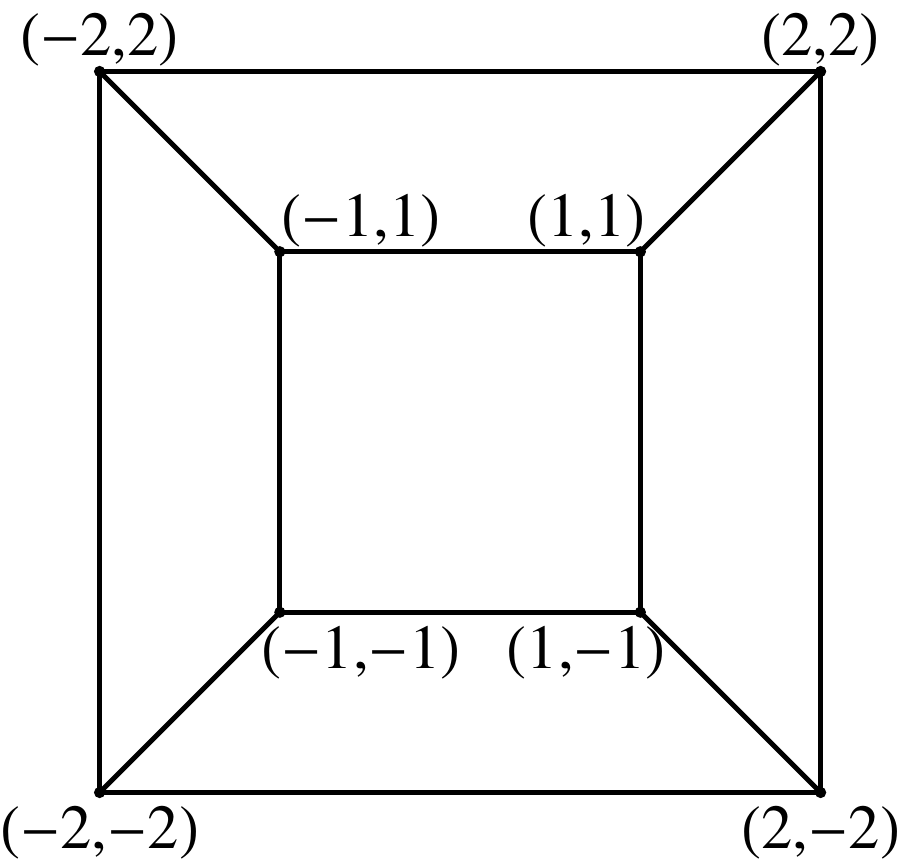}
\captionof{figure}{$\QC$}\label{SEVL}
\end{minipage}
\begin{minipage}[h]{.5\textwidth}
It is readily verifiable that the constant function $\mathbf{1}\in C^0(\QC)$ cannot be written as a sum of splines which are supported on the stars of the $4$ interior vertices, i.e. splines which restrict to $0$ outside of the shaded regions in Figure~\ref{cx}.
\end{minipage}
\begin{figure}[htp]
%\captionsetup[subfigure]{labelformat=empty}
\centering
\subfloat{\includegraphics[scale=.25]{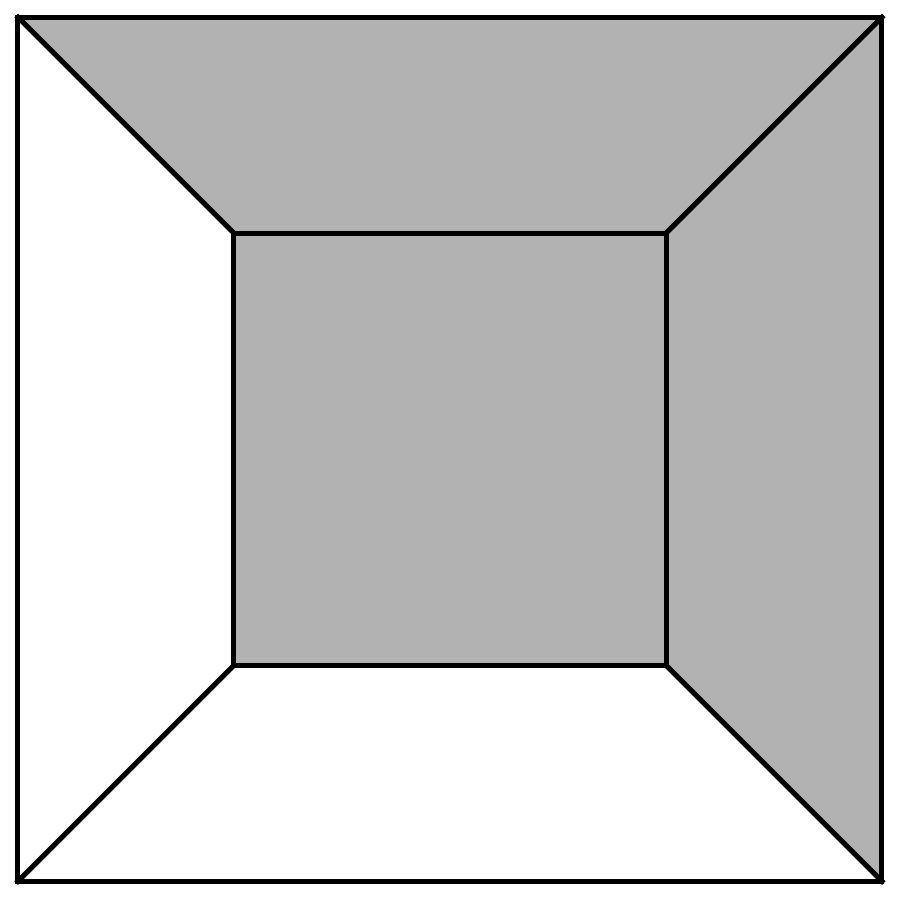}
}
\subfloat{
\includegraphics[scale=.25]{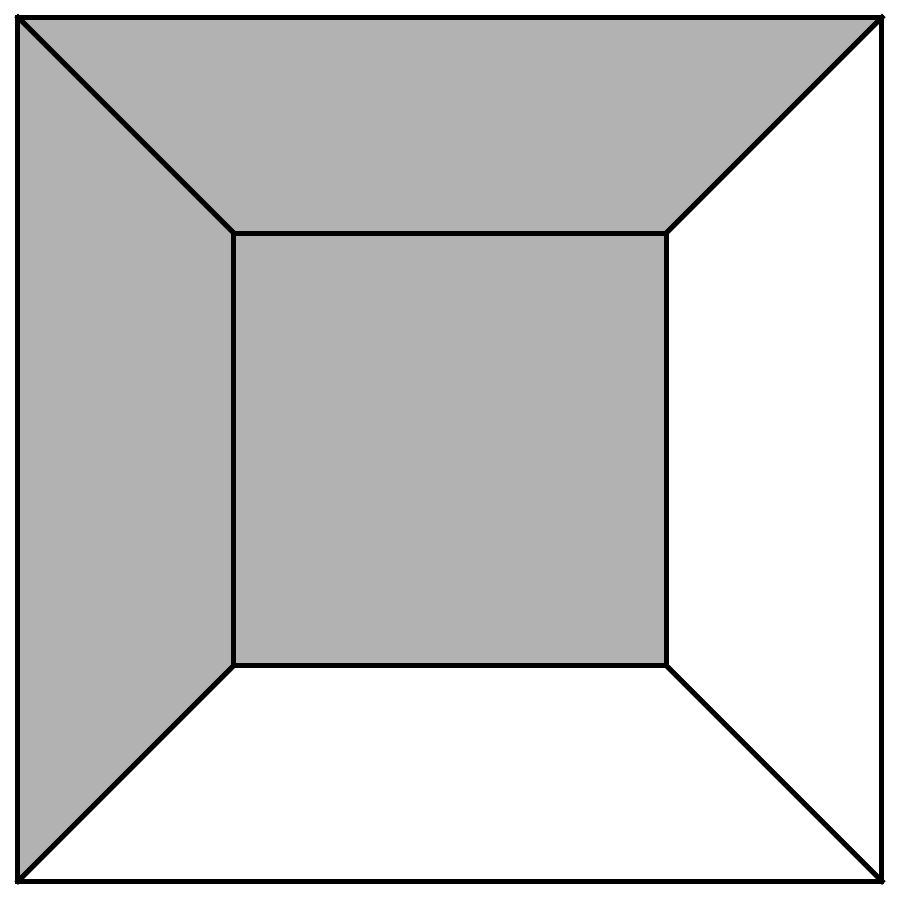}
}
\subfloat{
\includegraphics[scale=.25]{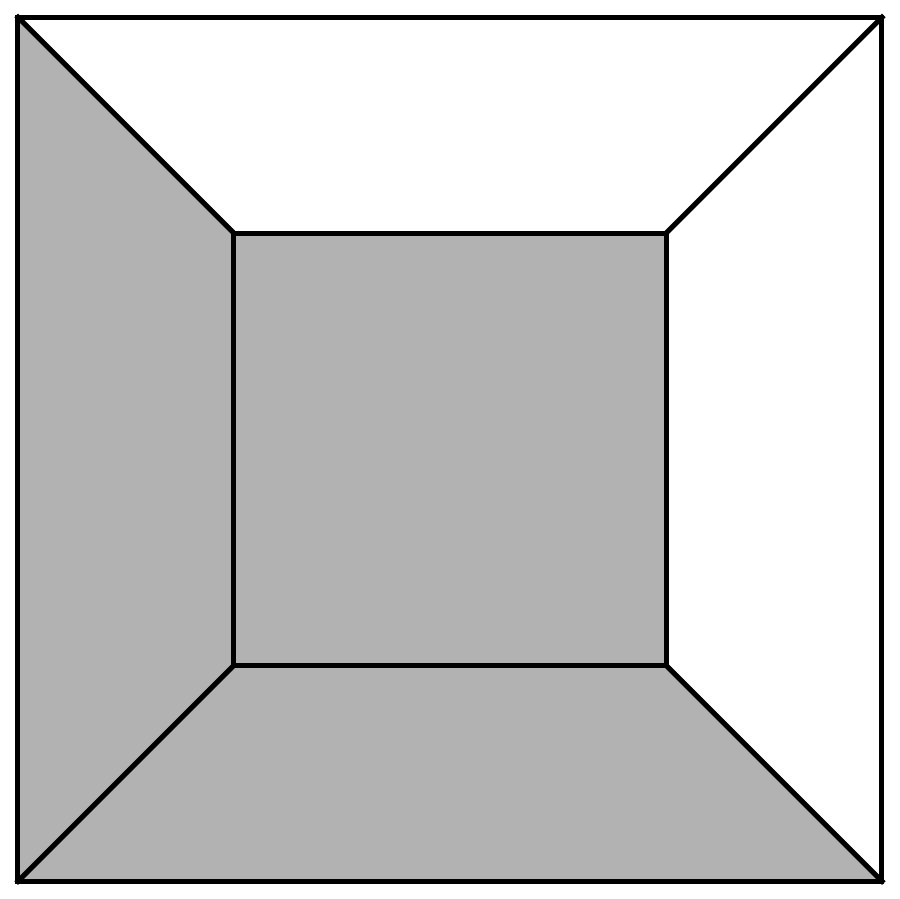}
}
\subfloat{
\includegraphics[scale=.25]{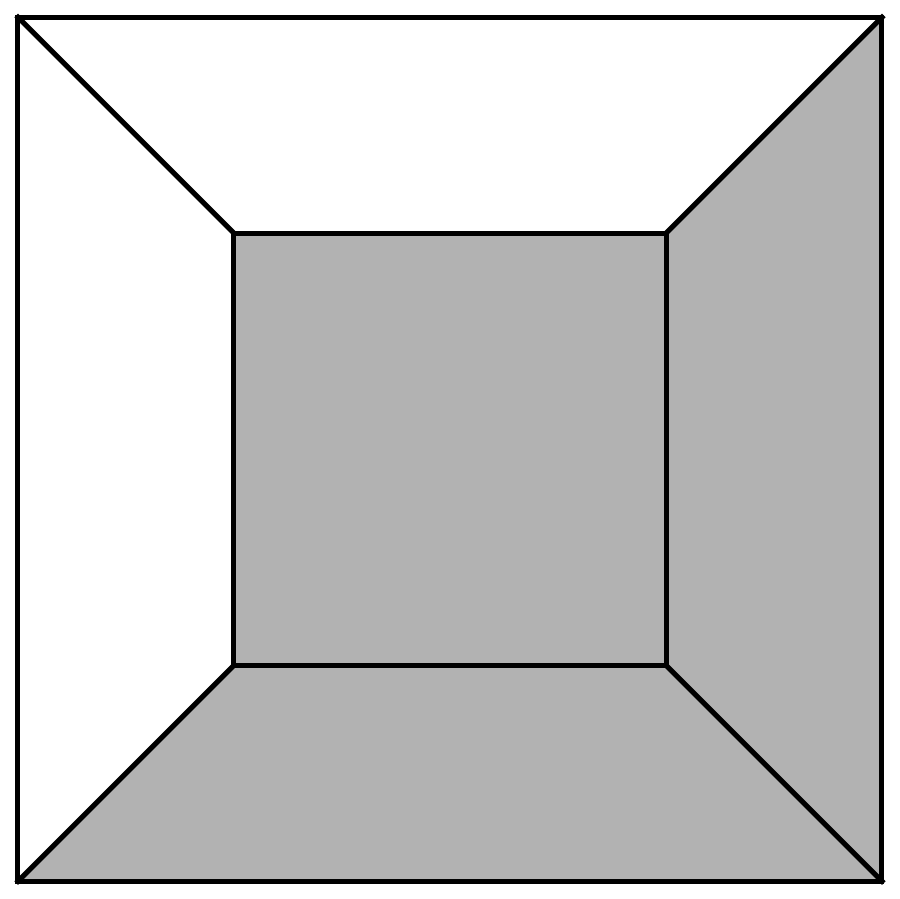}
}
\caption{Stars of Interior Vertices of $\QC$}\label{cx}
\end{figure}

However, Macaulay2 \cite{M2} computes the following decomposition of $\mathbf{1}\in C^0_2(\QC)$.

\begin{figure}[htp]
\begin{flushleft}
\begin{minipage}{.25\textwidth}
\centering
\includegraphics[scale=.3]{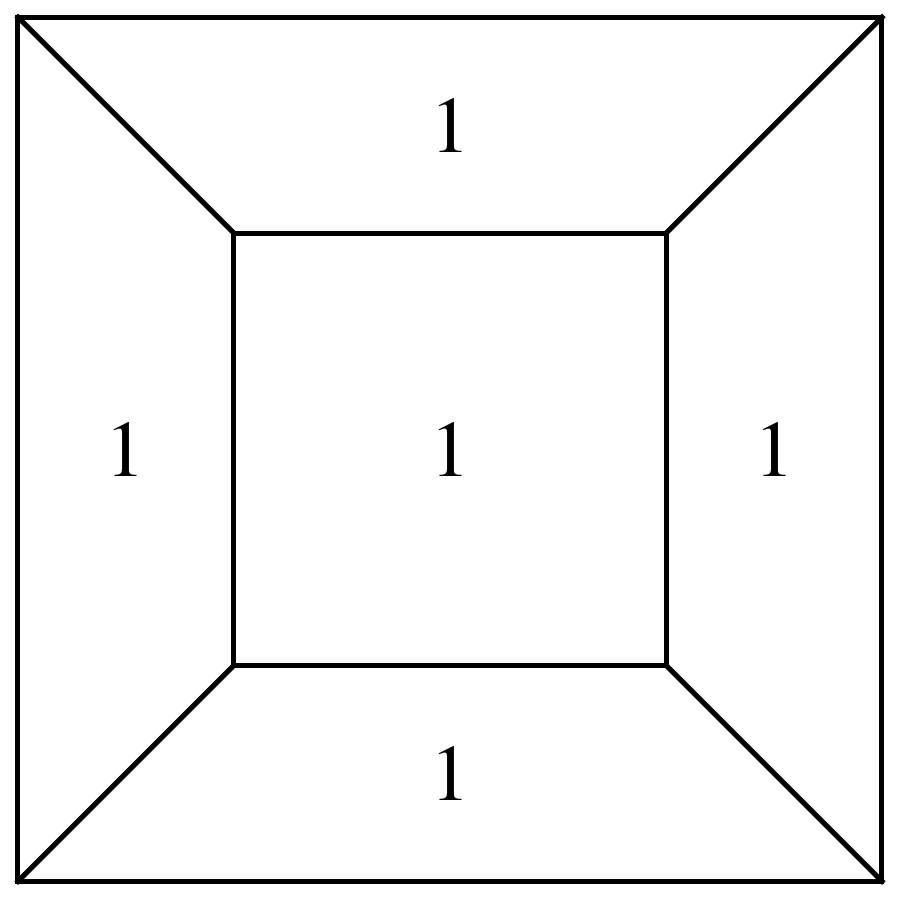}
\end{minipage}
\begin{minipage}{.05\textwidth}
\centering
$=$
\end{minipage}
\begin{minipage}{.25\textwidth}
\centering
\includegraphics[scale=.3]{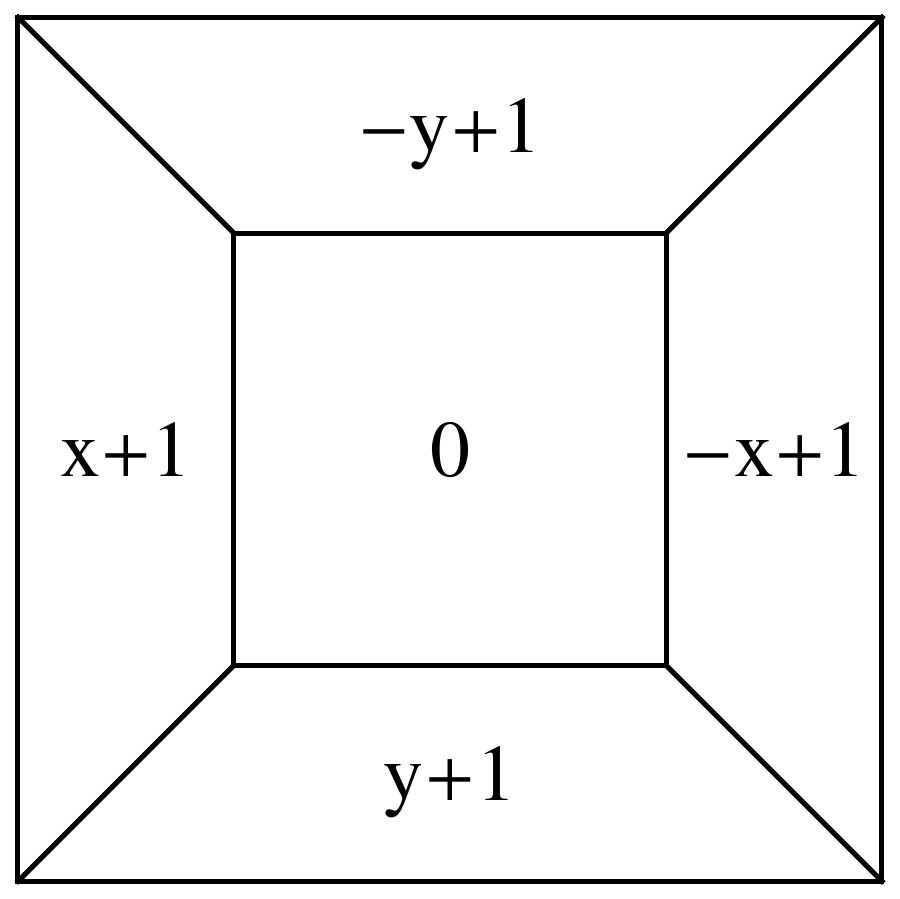}
\end{minipage}
\begin{minipage}{.05\textwidth}
\centering
$+$
\end{minipage}
\begin{minipage}{.25\textwidth}
\centering
\includegraphics[scale=.3]{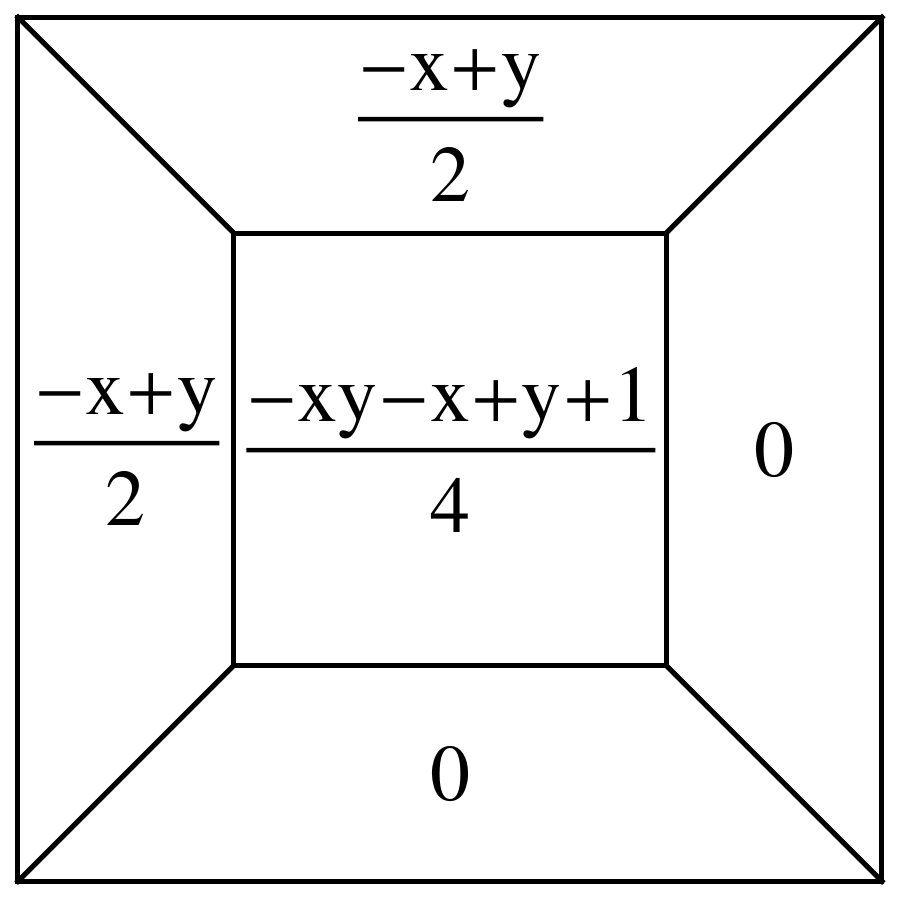}
\end{minipage}
\begin{minipage}{.05\textwidth}
\centering
$+$
\end{minipage}

\begin{minipage}{.25\textwidth}
\centering
\includegraphics[scale=.3]{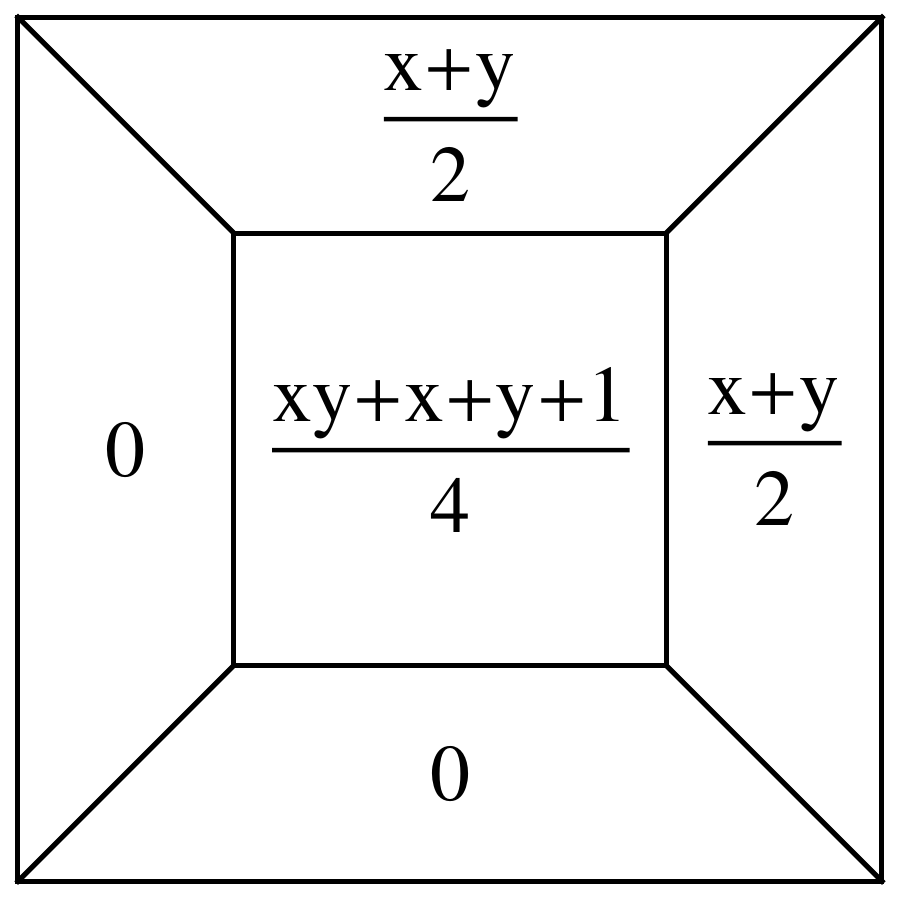}
\end{minipage}
\begin{minipage}{.05\textwidth}
\centering
$+$
\end{minipage}
\begin{minipage}{.25\textwidth}
\centering
\includegraphics[scale=.3]{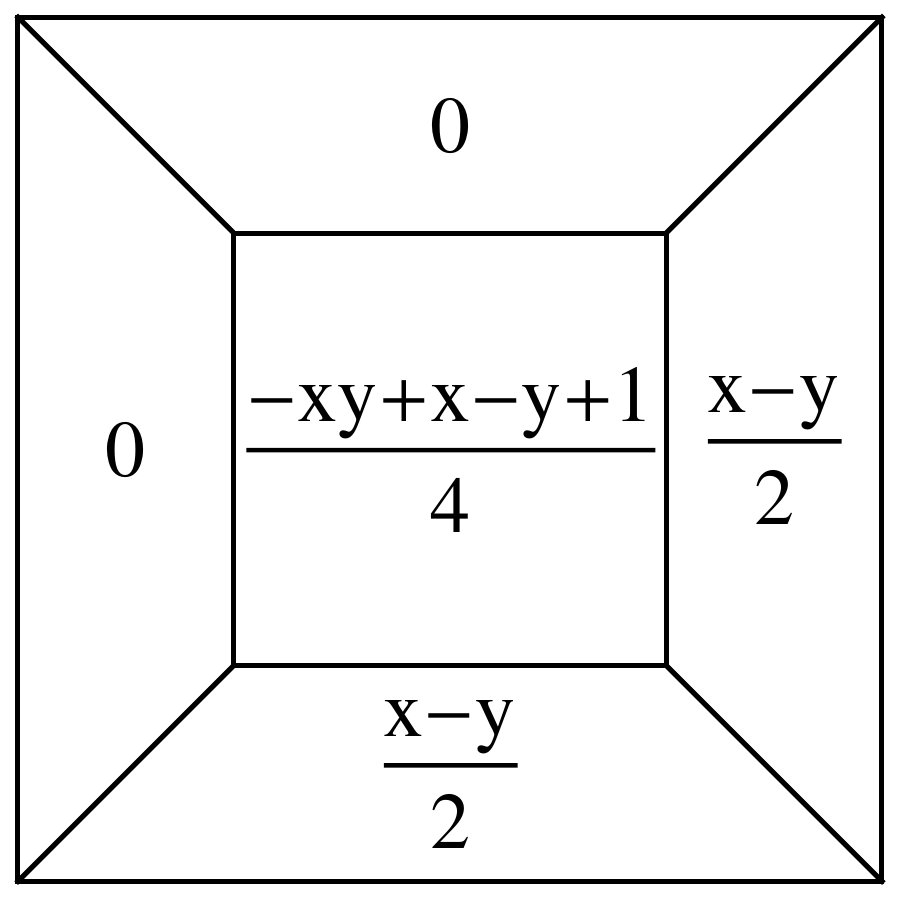}
\end{minipage}
\begin{minipage}{.05\textwidth}
\centering
$+$
\end{minipage}
\begin{minipage}{.25\textwidth}
\centering
\includegraphics[scale=.3]{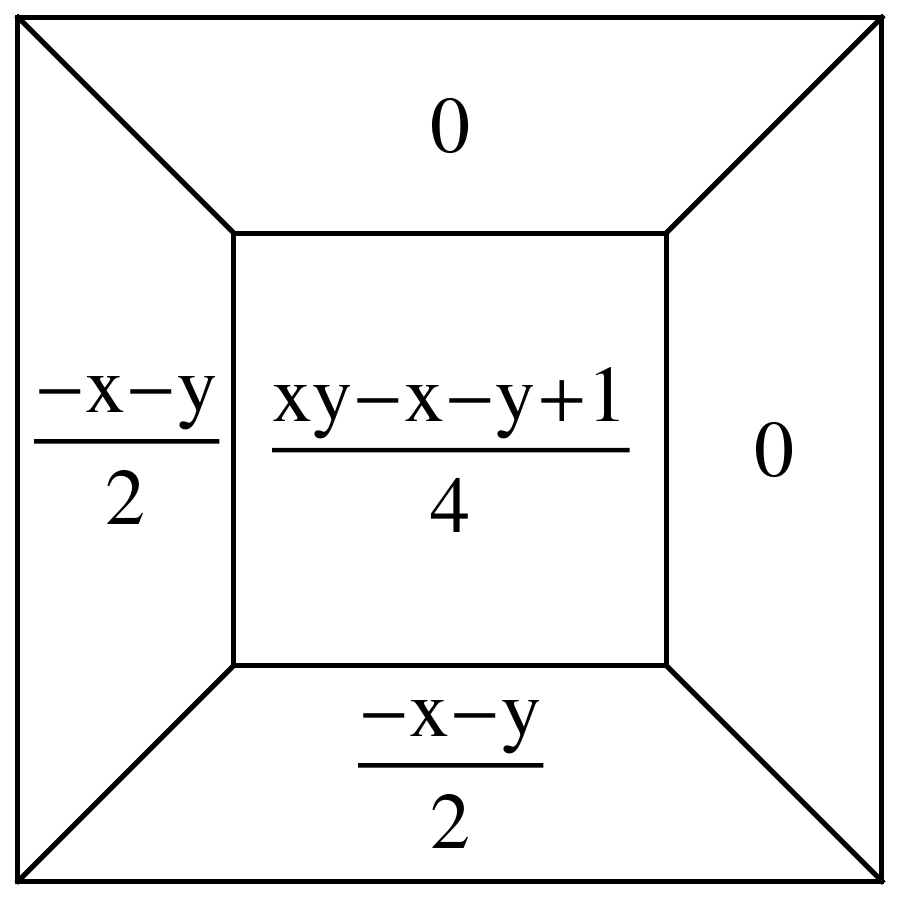}
\end{minipage}
\end{flushleft}
\caption{A `local' decomposition of $\mathbf{1}\in C^2_0(\QC)$}\label{1decomp}
\end{figure}

The support of the first spline in this sum is the annular subcomplex in Figure~\ref{xi}.  According to Theorem~\ref{main}, $C^r_d(\QC)$ has a basis of splines supported on either the star of an interior vertex or one of the shaded complexes in Figure~\ref{xi}, for $d\gg 0$ (See Example~\ref{Gam}).  We call such complexes \textit{lattice complexes}, explained in $\S 2$.

\begin{figure}[htp]
%\captionsetup[subfigure]{labelformat=empty}
\centering
\subfloat{
\includegraphics[scale=.23]{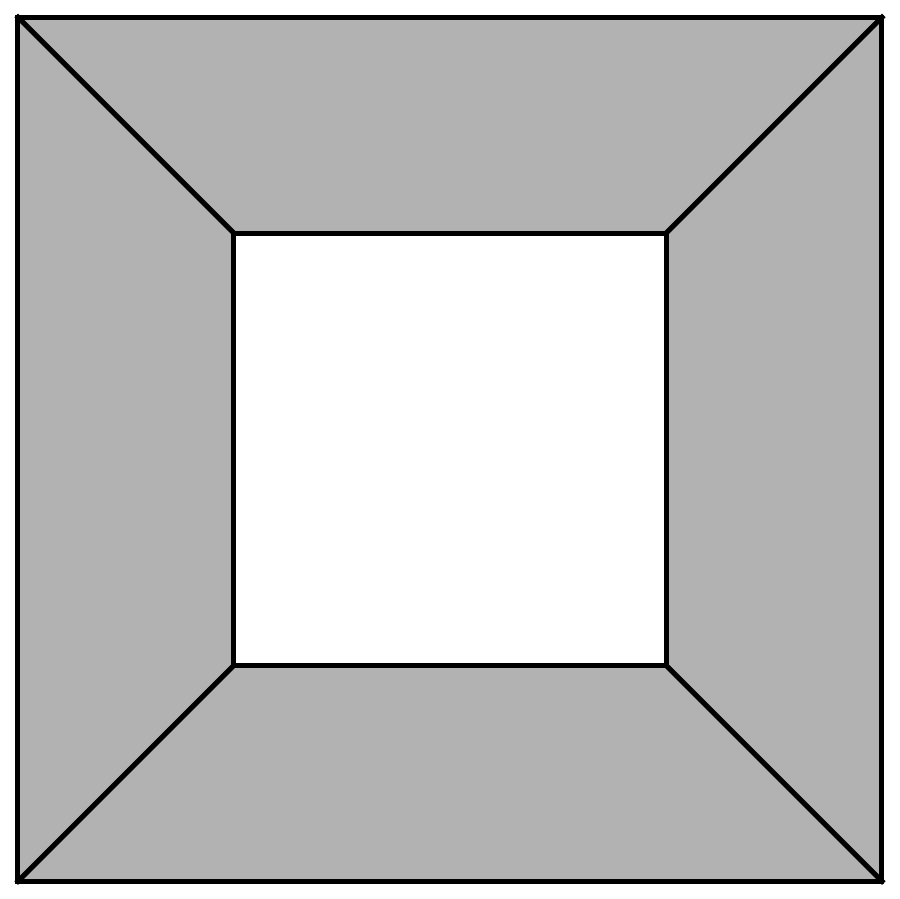}
}
\subfloat{
\includegraphics[scale=.23]{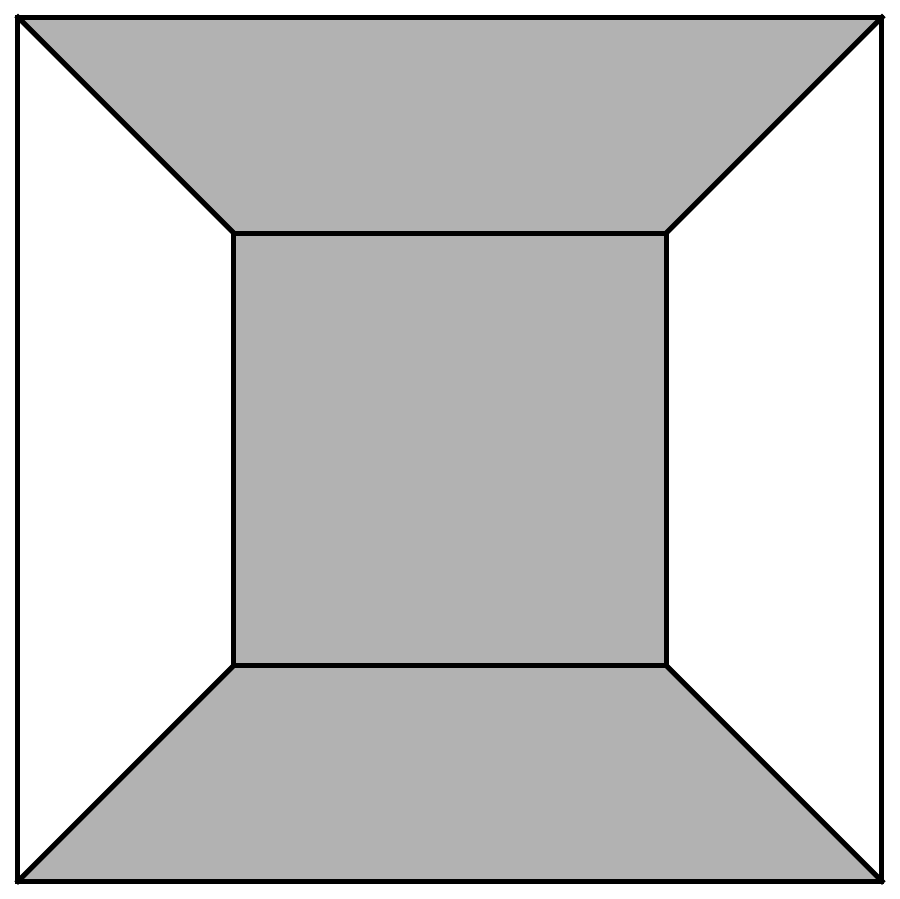}
}
\subfloat{
\includegraphics[scale=.23]{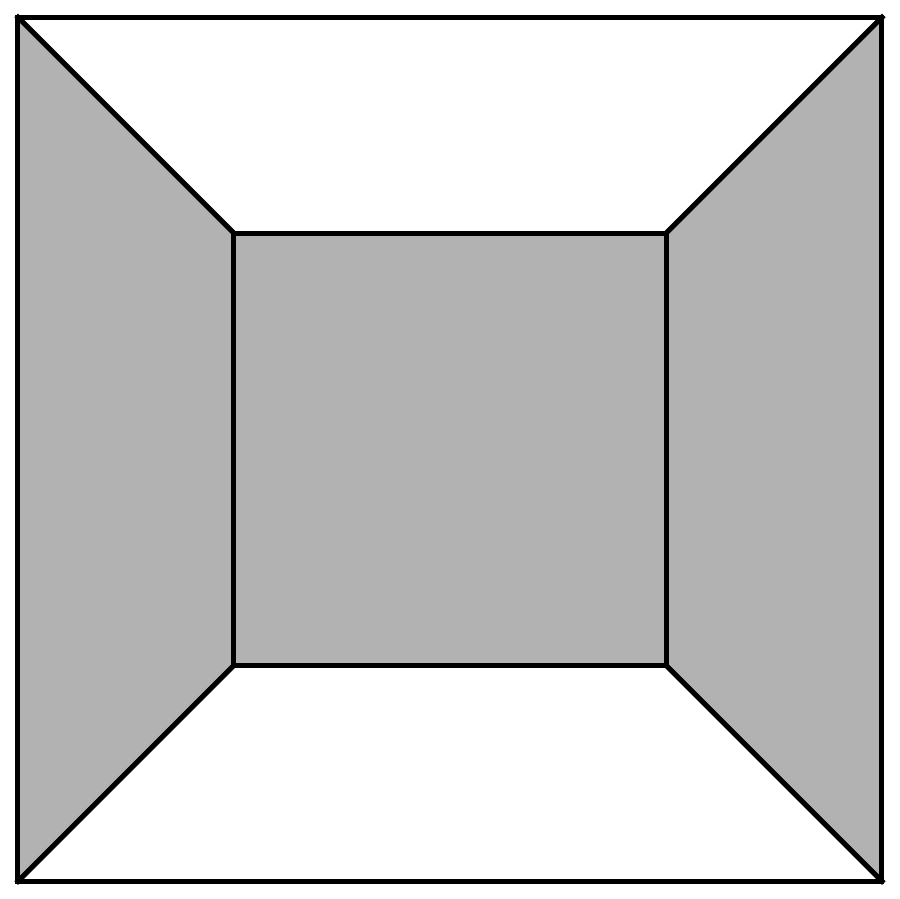}
}
\caption{Lattice Complexes of $\QC$}\label{xi}
\end{figure}

\begin{flushleft}
It is the presence of such an annular subcomplex in Figure~\ref{xi} that contributes to the constant term of the dimension formula $\mbox{dim}_\R C^r_d(\QC)$ for $d\gg 0$ provided in \cite{TSchenck08}.  So the lattice complexes describe in $\S 2$ encode subtle interactions between the geometry and combinatorics of $\PC$ that are manifested in $C^r(\PC)$.
\end{flushleft}

\begin{minipage}[h]{.4\textwidth}
\centering
\includegraphics[scale=.45]{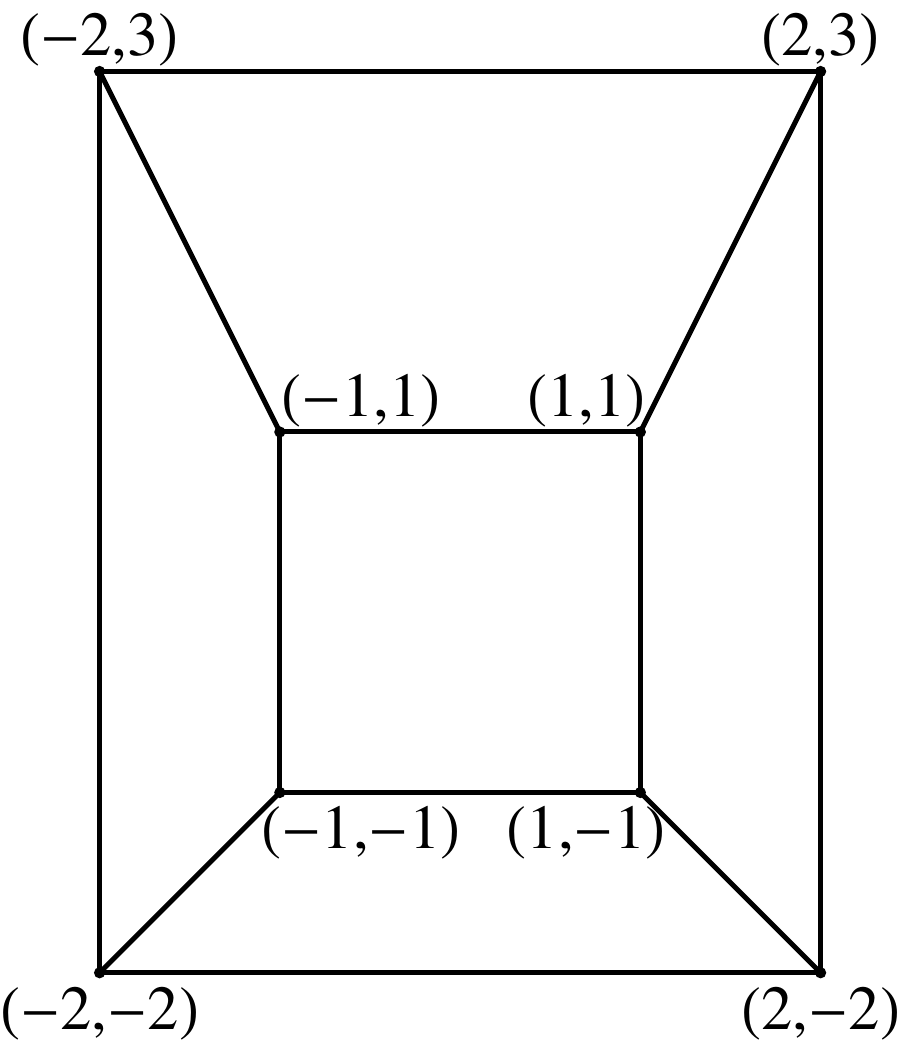}
\captionof{figure}{$\QC'$}\label{DSEVL}
\end{minipage}
\begin{minipage}[h]{.5\textwidth}
If we disturb the symmetry of $\QC$ slightly to get $\QC'$ as in Figure~\ref{DSEVL}, then the affine spans of the four edges connecting the inner and outer squares do not all intersect at the same point.  By Theorem~\ref{main}, $C^r_d(\QC')$ has a basis of splines with support in either the star of an interior vertex or one of the shaded complexes in Figure~\ref{dcx}, for $d\gg 0$ (see Example~\ref{Gam}).
\end{minipage}

\begin{figure}[htp]
\subfloat{
\includegraphics[scale=.3]{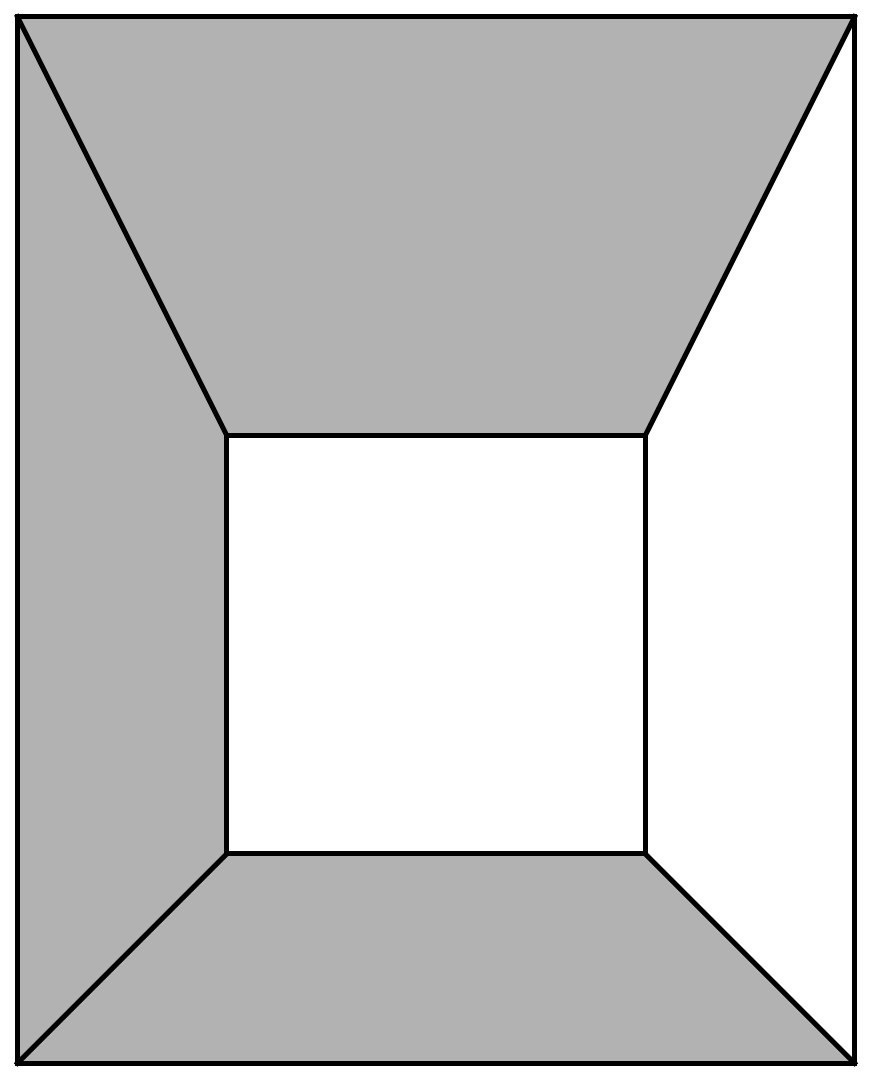}
}
\subfloat{\includegraphics[scale=.3]{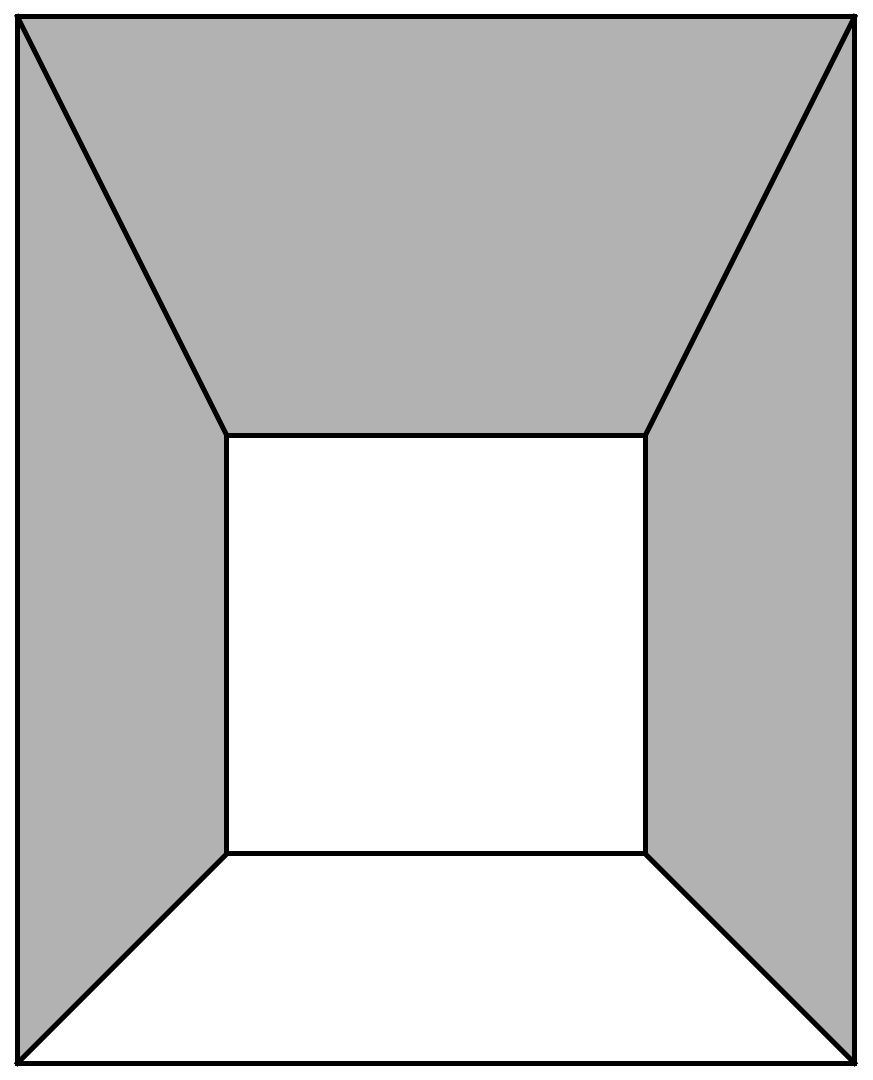}
}
\subfloat{
\includegraphics[scale=.3]{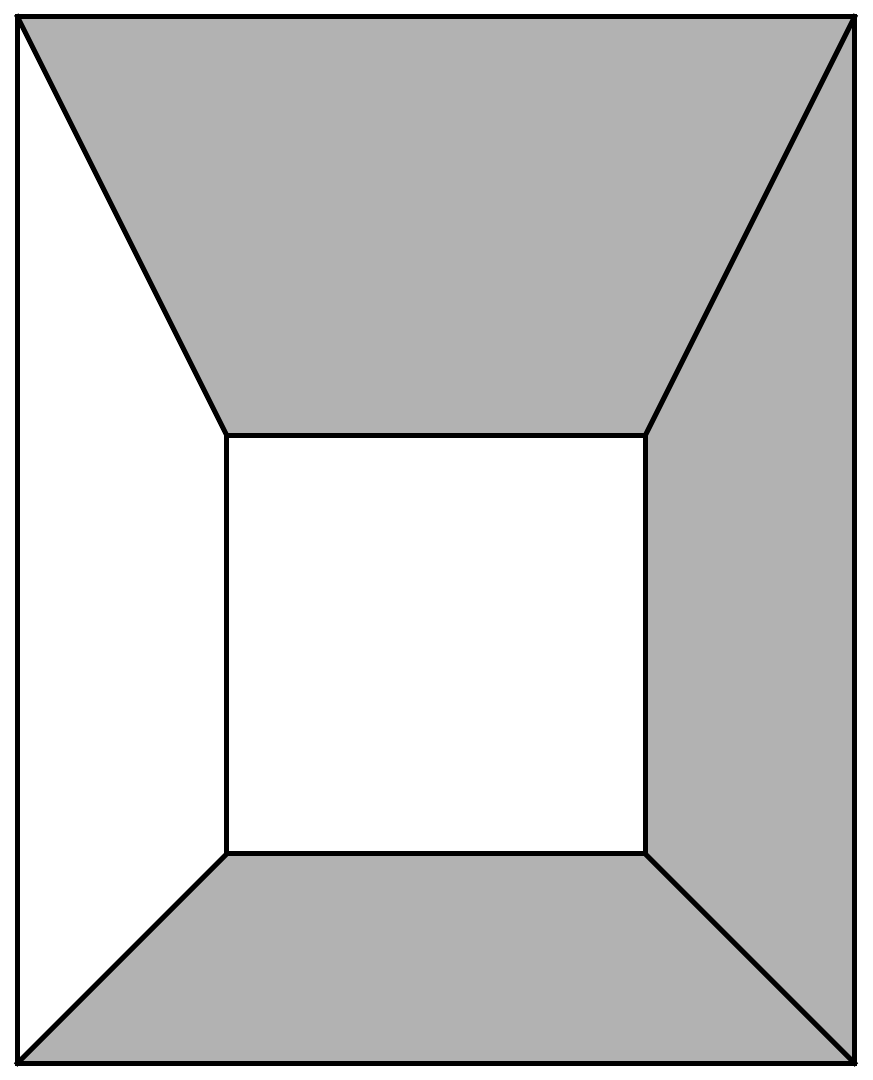}
}

\subfloat{
\includegraphics[scale=.3]{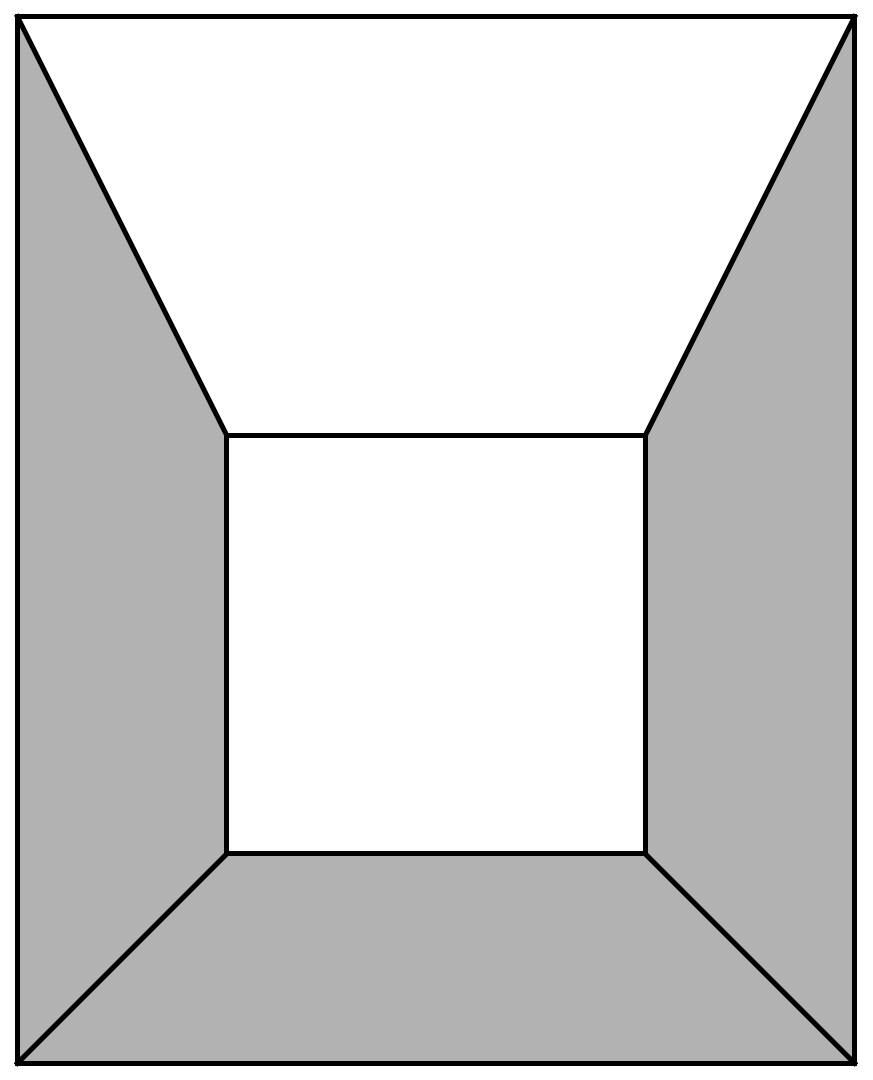}
}
\subfloat{
\includegraphics[scale=.3]{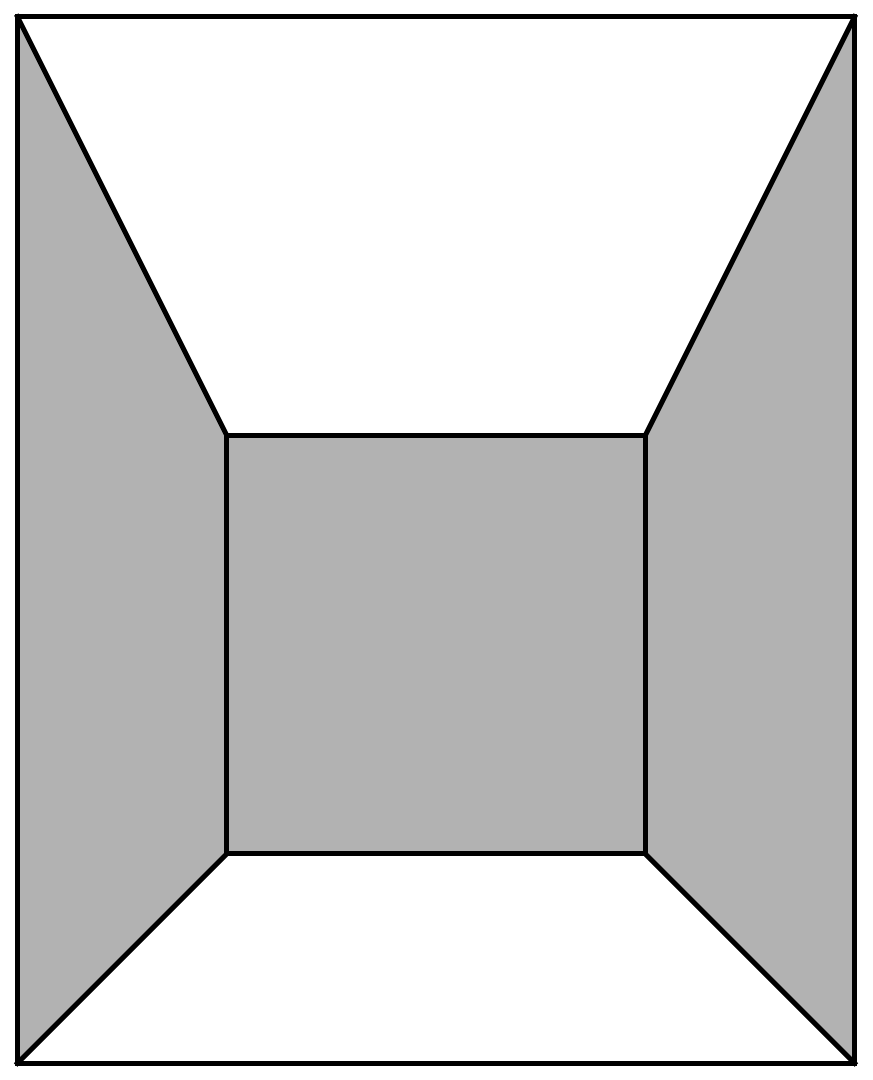}
}
\subfloat{
\includegraphics[scale=.3]{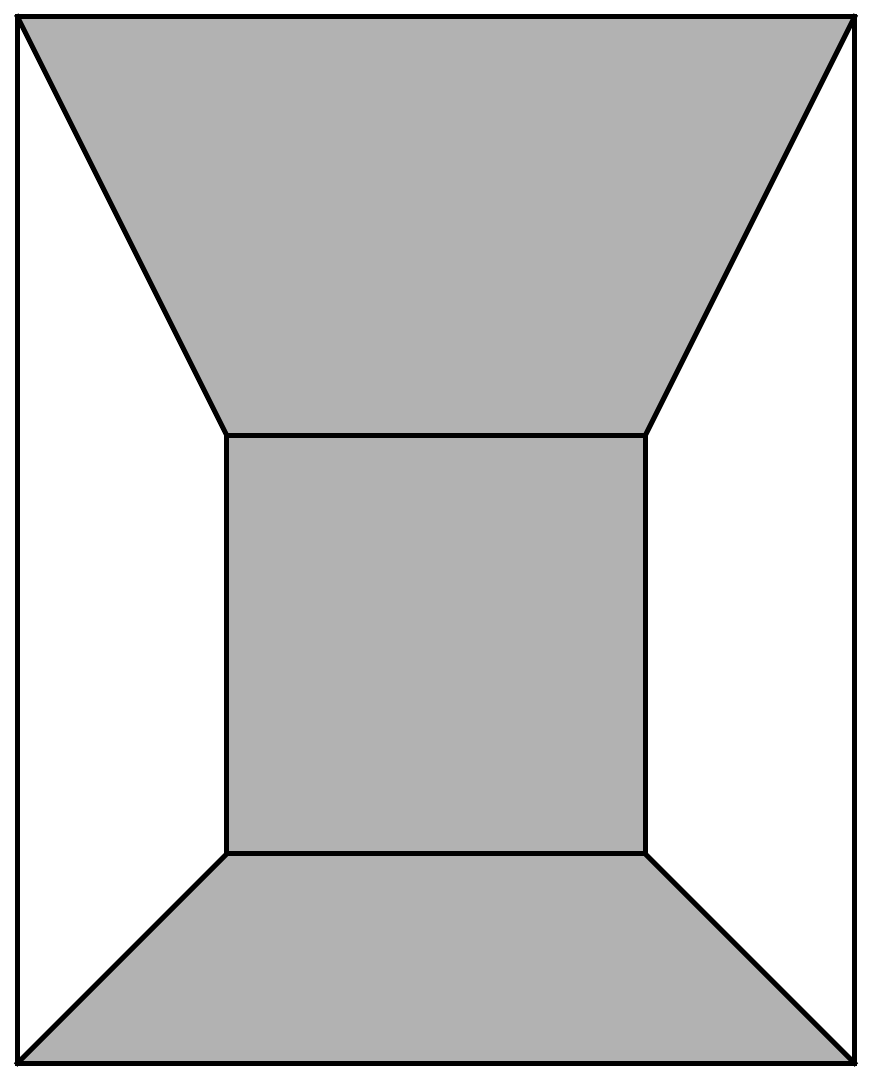}
}
\caption{Lattice Complexes of $\QC'$}\label{dcx}
\end{figure}

\section{Lattice Complexes}
We begin with some preliminary notions.  A \textit{polytopal complex} $\PC\subset\R^n$ is a finite set of convex polytopes (called \textit{faces} of $\PC$) in $\R^n$ such that
\begin{itemize}
\item If $\gamma\in\PC$, then all faces of $\gamma$ are in $\PC$.
\item If $\gamma,\tau\in\PC$ then $\gamma\cap\tau$ is a face of both $\gamma$ and $\tau$ (possibly empty).
\end{itemize}
The \textit{dimension} of $\PC$ is the greatest dimension of a face of $\PC$. 
The faces of $\PC$ are ordered via inclusion; a maximal face of $\PC$ is 
called a \textit{facet} of $\PC$, and $\PC$ is said to be \textit{pure} if all facets are equidimensional.  $|\PC|$ denotes the underlying space of $\PC$. $\PC_i$ and $\PC^0_i$ denote the set of $i$-faces and the set of interior $i$-faces, respectively.  In the case that all facets of $\PC$ are simplices, $\PC$ is a simplicial complex and will be denoted by $\Delta$.

Given a complex $\PC$ and a face $\gamma\in\PC$, the \textit{star} of $\gamma$ in $\PC$, denoted $\mbox{st}_\PC  (\gamma)$, is defined by
\[
\mbox{st}_\PC (\gamma):=\{\psi\in\PC|\exists\sigma\in\PC, \psi\in\sigma,\gamma\in\sigma \}.
\]
This is the smallest subcomplex of $\PC$ which contains all faces which contain $\gamma$.  If the complex $\PC$ is understood we will write $\mbox{st}(\gamma)$.

For $\PC\subset\R^n$, $G(\PC)$ is a graph, with a vertex for every facet (element of $\PC_n$); two vertices are joined by an edge iff the corresponding facets $\sigma$ and $\sigma'$ satisfy $\sigma\cap\sigma'\in\PC_{n-1}$.  $\PC$ is said to be \textit{hereditary} if $G(\mbox{st}_\PC(\gamma))$ is connected for every nonempty $\gamma\in\PC$.  Throughout this paper, $\PC\subset\R^n$ is assumed to be a pure, $n$-dimensional, hereditary polytopal complex.

Let $R=\R[x_1,\ldots,x_n]$ be the polynomial ring in $n$ variables.  For a subset $S\subset\R^n$, let $I(S)\subset R$ denote the ideal of polynomials vanishing on $S$.  If $\tau\in\PC_{n-1}$ then $l_\tau$ denotes any linear form generating the principal ideal $I(\tau)$.

In what follows we use a subgraph $G_J(\PC)$ of $G(\PC)$ determined by an ideal $J\subset R$.  This is a slight generalization of a graph used by McDonald and Schenck in \cite{TSchenck08}.

\begin{defn}\label{graph} 
Let $\PC\subset\R^n$ be a polyhedral complex, and $J\subset R$ an ideal.  The vertices of the graph $G_J(\PC)$ correspond to facets of $\PC$ having a codimension one face $\tau$ such that $l_\tau\in J$.  Two vertices corresponding to facets $\sigma_1$ and $\sigma_2$ which intersect along the edge $\tau$ are connected in $G_J(\PC)$ if $l_\tau\in J$.
\end{defn}

Let $G_J(\PC)$ be the union of $k$ connected components $G^1_J(\PC),\ldots,G^k_J(\PC)$.  There is a unique subcomplex $\PC^i_J$ of $\PC$ whose dual graph is $G^i_J(\PC)$.

\begin{defn}
With notation as above, define $\PC_J$ to be the \textbf{disjoint} union of $\PC^1_J,\ldots,\PC^k_J$.
\end{defn}

We call the $\PC^i_J$ the connected components, or simply components, of $\PC_J$, although two components may share codimension one faces within $\PC$ as will be apparent in Example~\ref{SchCube}.  There are only finitely many distinct complexes $\PC_J$ associated to $J\subset R$.  They are in bijection with the nontrivial elements of the intersection poset of a certain hyperplane arrangement, which we now describe.

Recall that a hyperplane arrangement $\mathcal{H}\subset \R^n$ is a finite set $\mathcal{H}=\{H_1,\ldots,H_k\}$ of hyperplanes.  The \textit{intersection poset} $L(\mathcal{H})$ of $\mathcal{H}$ includes the whole space, the hyperplanes $H_i$, and nonempty intersections of these hyperplanes (called \textit{flats}) ordered with respect to reverse inclusion.  $L(\mathcal{H})$ is a \textit{ranked} poset with rank function $\mbox{rk}(W)=\mbox{codim}(W)$ for $W\in L(\mathcal{H})$.  $L(\mathcal{H})$ is a \textit{meet} semilattice and is a lattice iff $\mathcal{H}$ is \textit{central}, that is, iff $\cap_i H_i\neq \emptyset$.

\begin{defn}
Let $\PC\subset\R^n$ be a polyhedral complex.
\begin{enumerate}
\item For $\tau\in\PC$ a face, $\mbox{aff}(\tau)$ denotes the linear (or affine) span of $\tau$.
\item $\A(\PC)$ denotes the hyperplane arrangement $\cup_{\tau\in\PC^0_{n-1}}\mbox{aff}(\tau)$.
\item $L_\PC$ denotes the intersection semi-lattice $L(\A(\PC))$ of $\A(\PC)$.
\end{enumerate}
\end{defn}

\begin{lem}\label{lat}
For every ideal $J\subset R$, there is a unique $W\in L_\PC$ so that $\PC_J=\PC_{I(W)}$.  Furthermore, the ideal $I(W)$ is minimal with respect to $\PC_{I(W)}=\PC_J$.
\end{lem}

\begin{proof}
Set $W=\cap_{\tau\in(\PC_J)^0_{n-1}} \mbox{aff}(\tau)$.  Clearly $\PC_{I(W)}=\PC_J$.  To prove minimality, let $Q$ be any ideal satisfying $\PC_Q=\PC_J$.  Then all codim $1$ faces $\tau\in(\PC_Q)^0_{n-1}$ satisfy $l_\tau\in Q$.  Since $(\PC_Q)^0_{n-1}=(\PC_J)^0_{n-1}$, $I(W)\subset Q$.  To show uniqueness, assume $V\in L_\PC$ and $\PC_{I(V)}=\PC_J$.  By minimality of $I(W)$, $I(W)\subset I(V)$, implying $V\subset W$.  If $V\subsetneq W$, then there is some $\tau\in\PC_{n-1}^0\setminus (\PC_J)^0_{n-1}$ so that $V\subset\mbox{aff}(\tau)$.  But then $\tau$ is an interior edge of $\PC_{I(V)}$ that is not an interior edge of $\PC_J$, a contradiction.  So $V=W$.
\end{proof}

For brevity, henceforth we write $G_S(\PC)$ and $\PC_S$ to denote $G_{I(S)}(\PC)$ and $\PC_{I(S)}$ for $S\subset R^n$.

\begin{exm}\label{SchCube}
The planar polytopal complex $\QC$ from the introduction is shown in Figure~\ref{latcxs}, along with its associated line arrangement $\A(\QC)$ and a representative sample of the complexes $\QC_W$ for $W\neq\emptyset\in L_\QC$.  We label the interior edges of $\QC$ by $1,\ldots,8$ and denote their affine spans by $L1,\ldots,L8$.  For each complex $\QC_W$ the facets are shaded and the corresponding flat $W$ is labelled.  If $\QC_W$ is the disjoint union of several subcomplexes $\QC^i_W$, we display these subcomplexes separately.

\begin{figure}[htp]
\captionsetup[subfigure]{labelformat=empty}
\centering
\subfloat[$\QC$]{
\includegraphics[scale=.6]{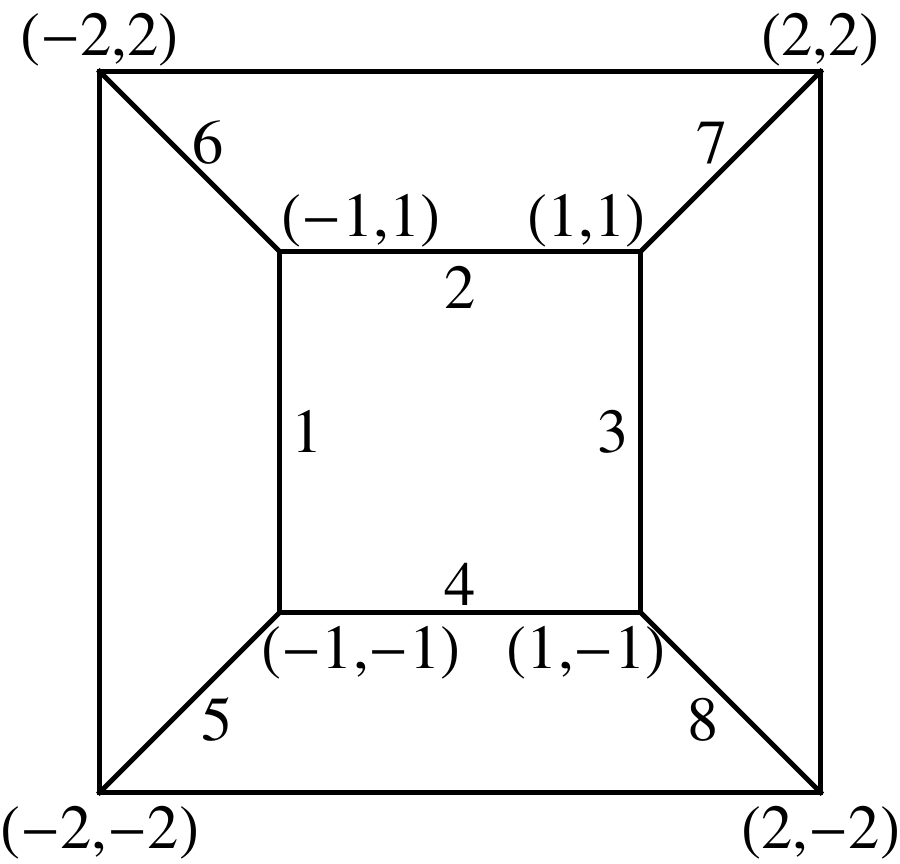}
}
\subfloat[$\A(\QC)$]{
\includegraphics[scale=.6]{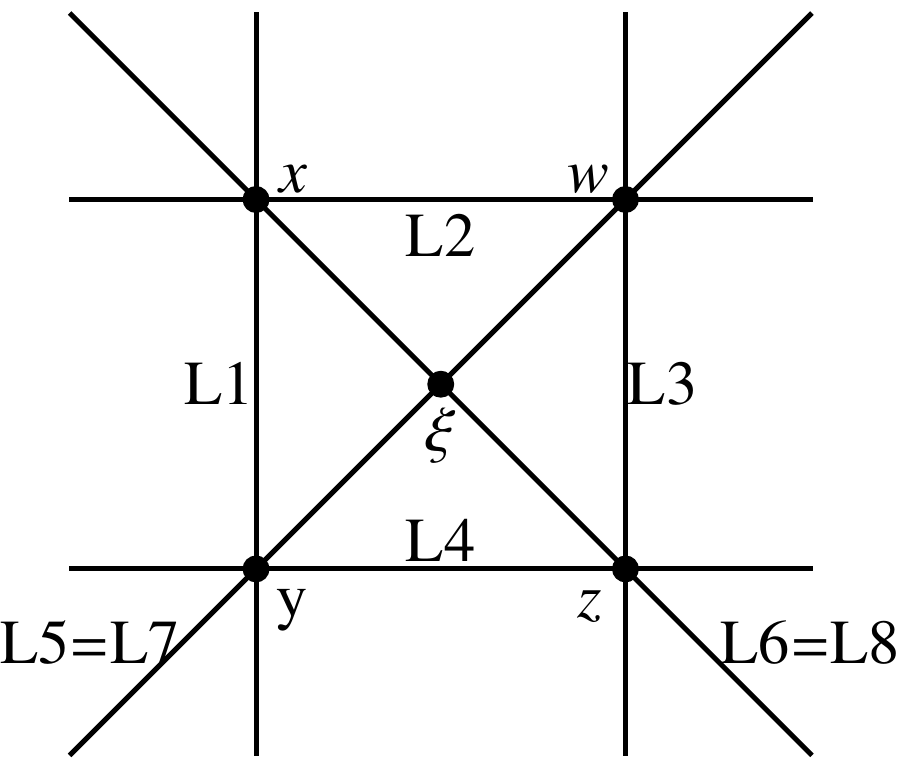}
}

\subfloat[$\QC_{L1}$]{
\includegraphics[scale=.4]{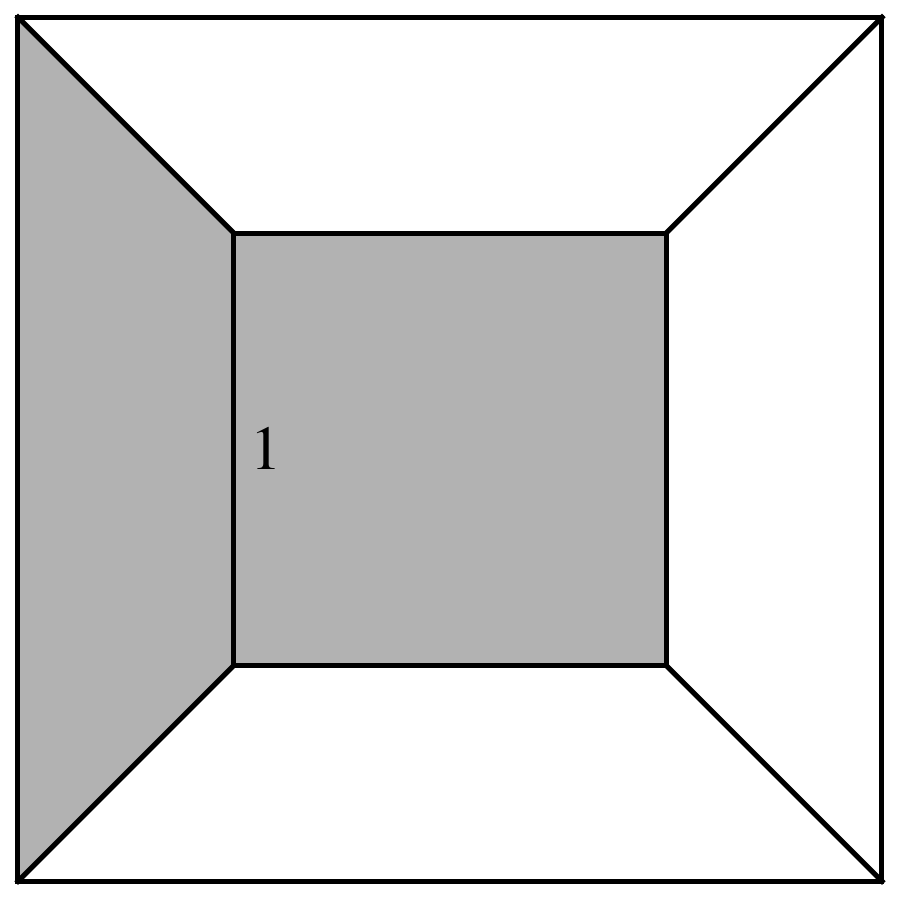}
}
\subfloat[$\QC^1_{L5}$]{
\includegraphics[scale=.4]{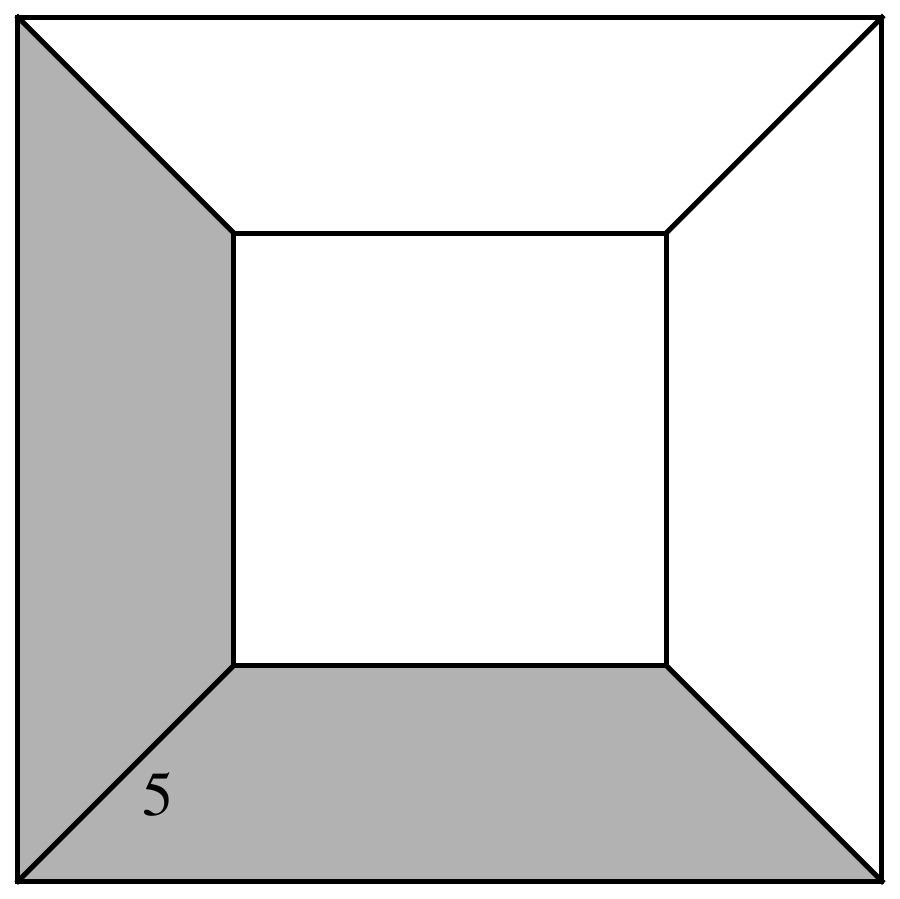}
}
\subfloat[$\QC^2_{L5}$]{
\includegraphics[scale=.4]{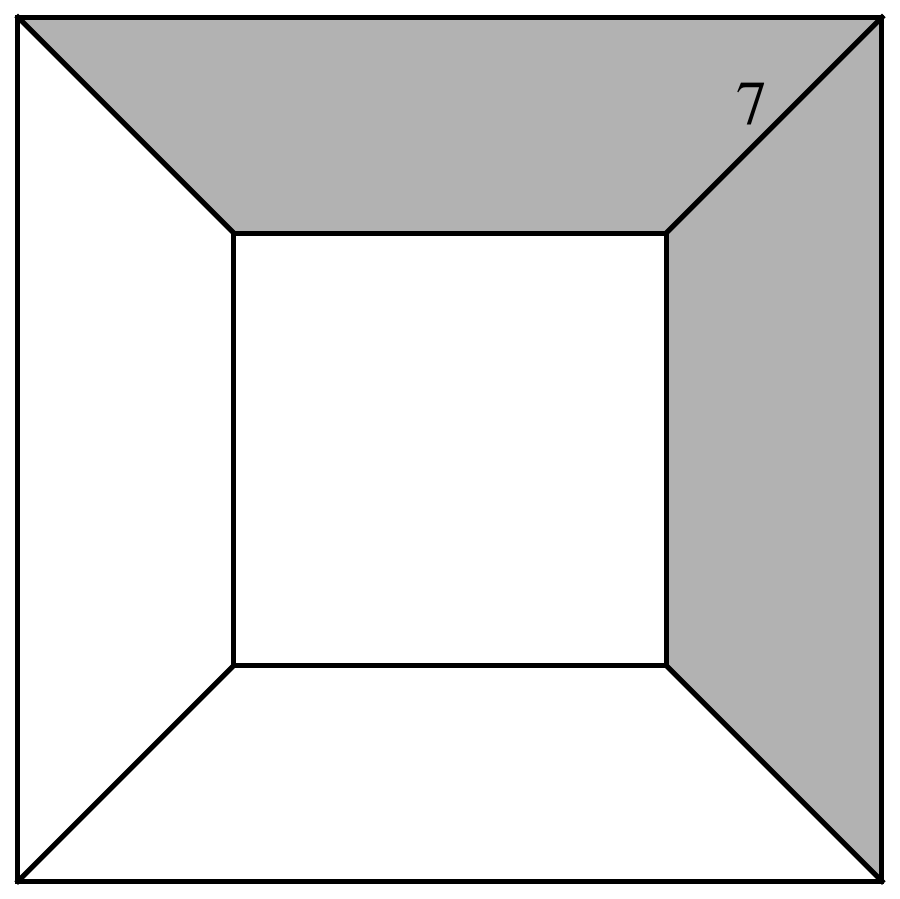}
}

\subfloat[$\QC^1_w$]{
\includegraphics[scale=.4]{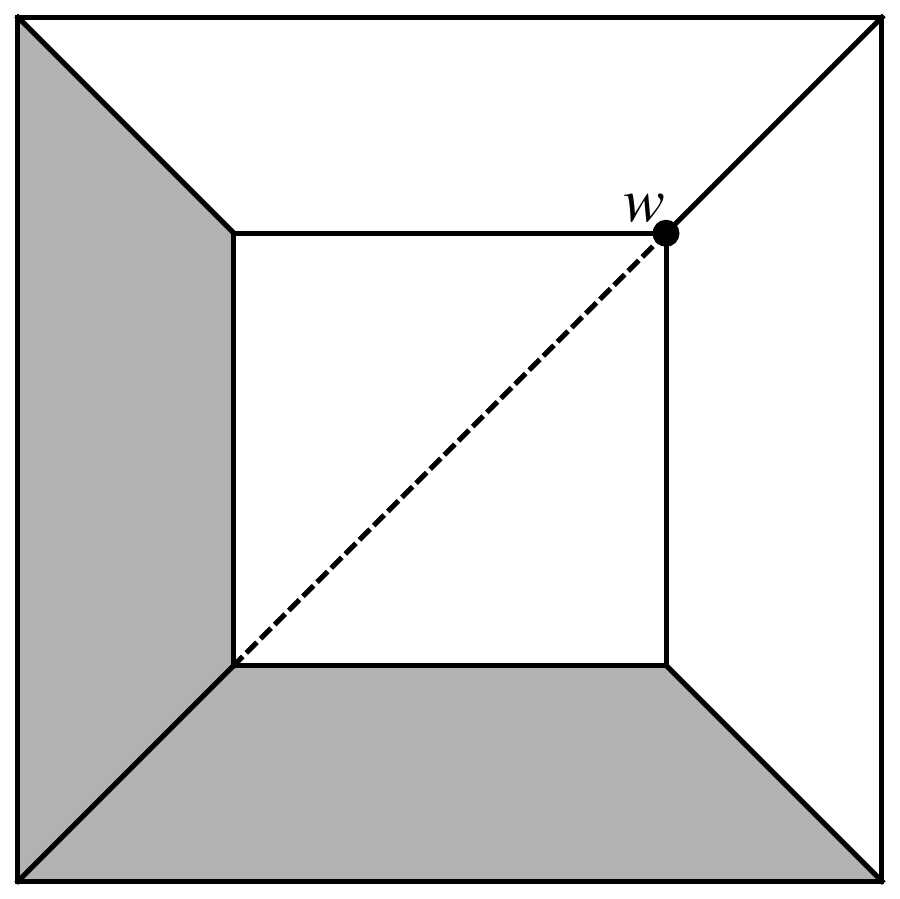}
}
\subfloat[$\QC^2_w$]{
\includegraphics[scale=.4]{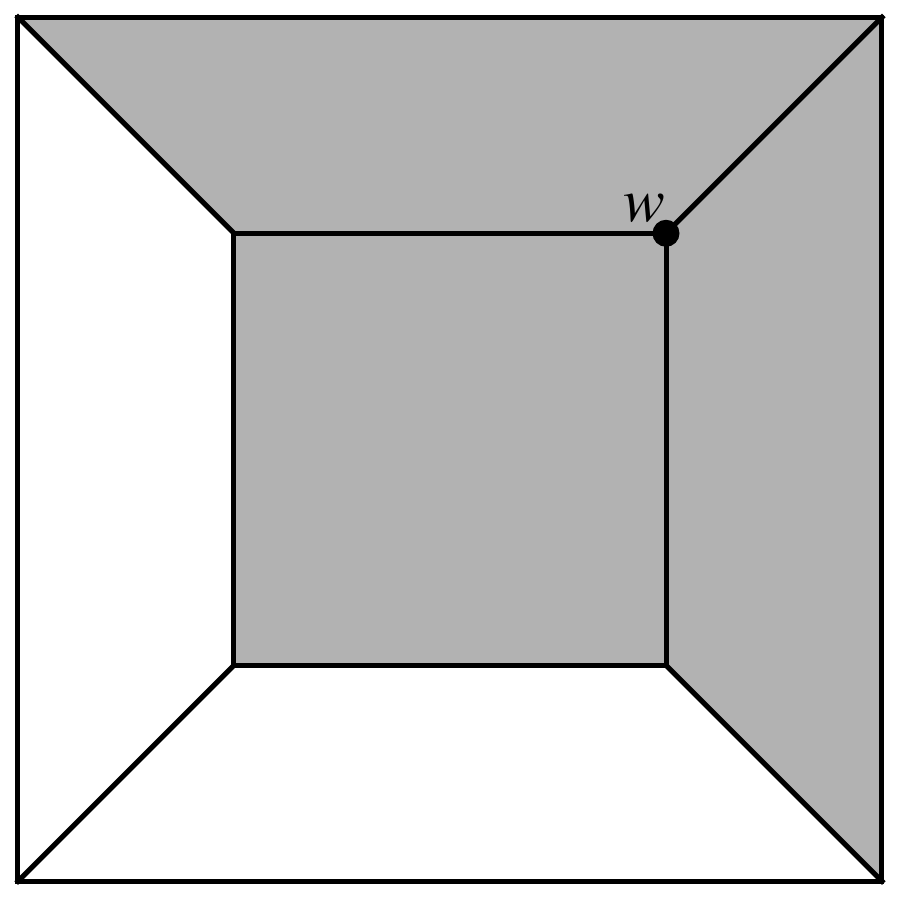}
}
\subfloat[$\QC_\xi$]{
\includegraphics[scale=.4]{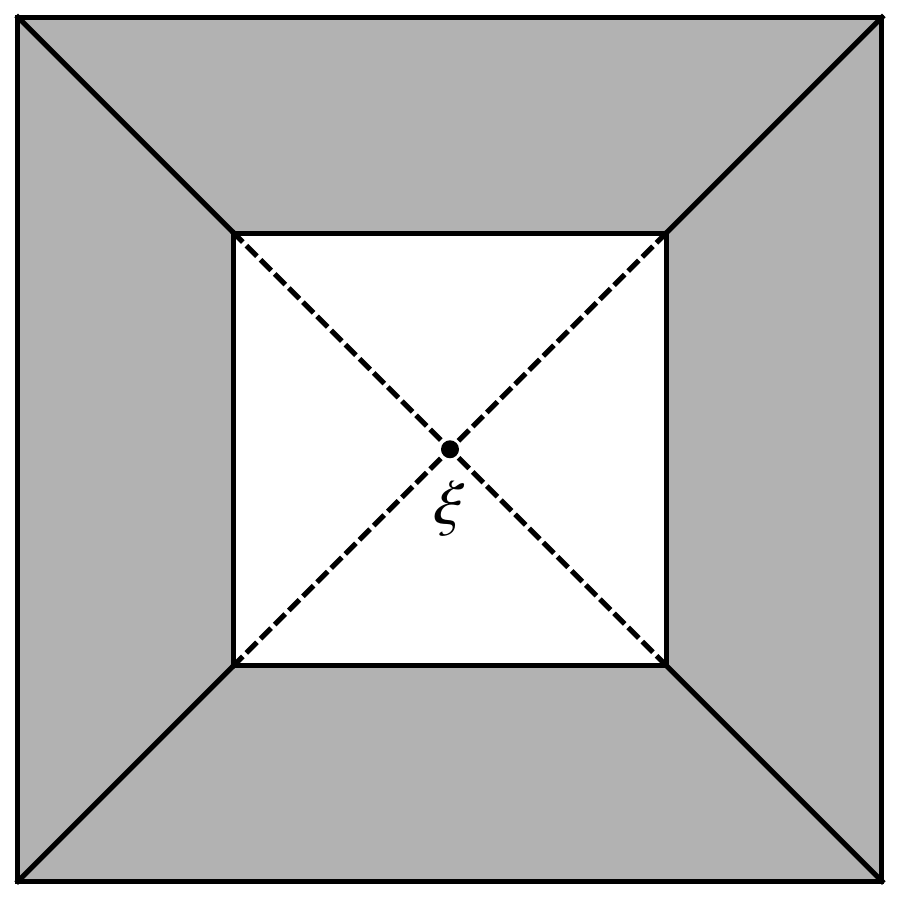}
}
\caption{Lattice Complexes of Example~\ref{SchCube}}\label{latcxs}
\end{figure}
\end{exm}

\subsection{The central case and homegenization}\label{homog}  The case where $\PC$ is a central complex, i.e. $\A(\PC)$ is a central arrangement, is of particular interest for splines since then $C^r(\PC)$ is a graded $R$-algebra.  All the hyperplanes of $\A(\PC)$ pass through the origin (perhaps after a coordinate change), so we can remove the origin and consider the projective arrangement $\mathbb{P}\A(\PC)\subset\mathbb{P}_\R^{n-1}$ obtained by quotienting under the action of $\R^*$ by scalar multiplication.  The intersection poset $L(\mathbb{P}\A(\PC))$ is identical to $L_\PC$ except it may not contain the maximal flat of $L_\PC$ (if that flat was the origin).

One of the most important central complexes for splines is the \textit{homogenization} $\widehat{\PC}\subset\R^{n+1}$ of a polytopal complex $\PC\subset \R^n$.  $\widehat{\PC}\subset\R^{n+1}$ is constructed by taking the join of $i(\PC)$ with the origin in $\R^{n+1}$, where $i:\R^n\rightarrow\R^{n+1}$ is defined by $i(a_1,\ldots,a_n)=(1,a_1,\ldots,a_n)$.  $C^r(\widehat{\PC})$ is a graded algebra whose $d$th graded piece $C^r(\widehat{\PC})_d$ is a vector space isomorphic to $C^r_d(\PC)$ (see \cite{DimSeries}).  If we regard $x_0,\ldots, x_n$ as coordinate functions on $\R^{n+1}$, then we obtain the original complex $\PC$ from $\widehat{\PC}$ by setting $x_0=1$.

\textbf{Remark}: We associate a subcomplex $\PC_W\subset \PC$ to $W\in L_{\widehat{\PC}}$ by slicing the complex $\widehat{\PC}_W$ with the hyperplane $x_0=1$.  Note that the subcomplex $\widehat{\PC}_W$ is the cone over the subcomplex $\PC_W$.  The subcomplexes $\PC_W\subset\PC$ obtained this way are the same as those obtained by first embedding $\PC$ in $\mathbb{P}^n_\R$ by adding the hyperplane at infinity, taking the arrangement of hyperplanes $\A(\PC)$ in $\mathbb{P}^n_\R$ (including intersections in the hyperplane at infinity), and forming the complexes $\PC_W$ for flats $W$ in this projective arrangement.

In Figure~\ref{projective} we show the arrangement $\mathbb{P}\A(\widehat{\QC})$ for the complex $\QC$ in Example~\ref{SchCube}.  The lattice $L_{\widehat{\QC}}$ has two rank $2$ flats $\alpha$ and $\beta$ which do not appear in $L_\QC$, corresponding to the intersections of the two pairs of parallel lines $L1,L3$ and $L2,L4$ in $\mathbb{P}^2$.  The complexes $\QC_\alpha,\QC_\beta$, also depicted in $\mathbb{P}^2$, are identical up to rotation.

\begin{figure}[htp]
\captionsetup[subfigure]{labelformat=empty}
\centering
\subfloat[$\mathbb{P}\A(\widehat{\QC})$]{
\includegraphics[scale=.45]{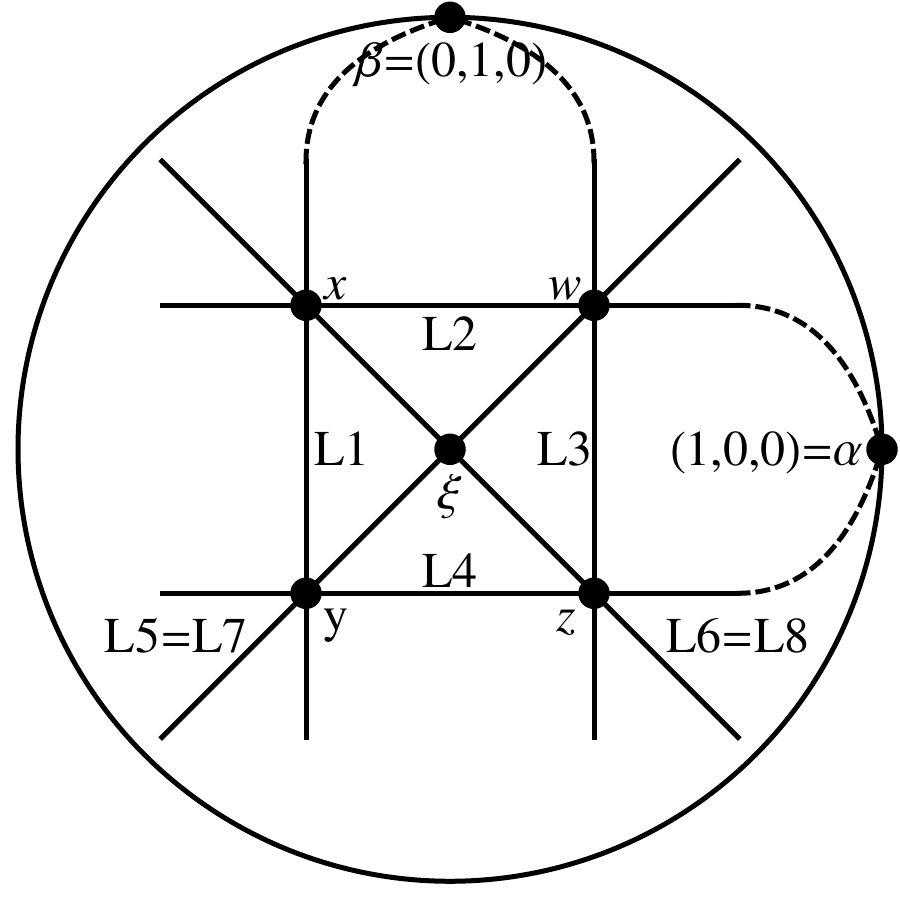}
}
\subfloat[$\QC_\alpha$]{
\includegraphics[scale=.4]{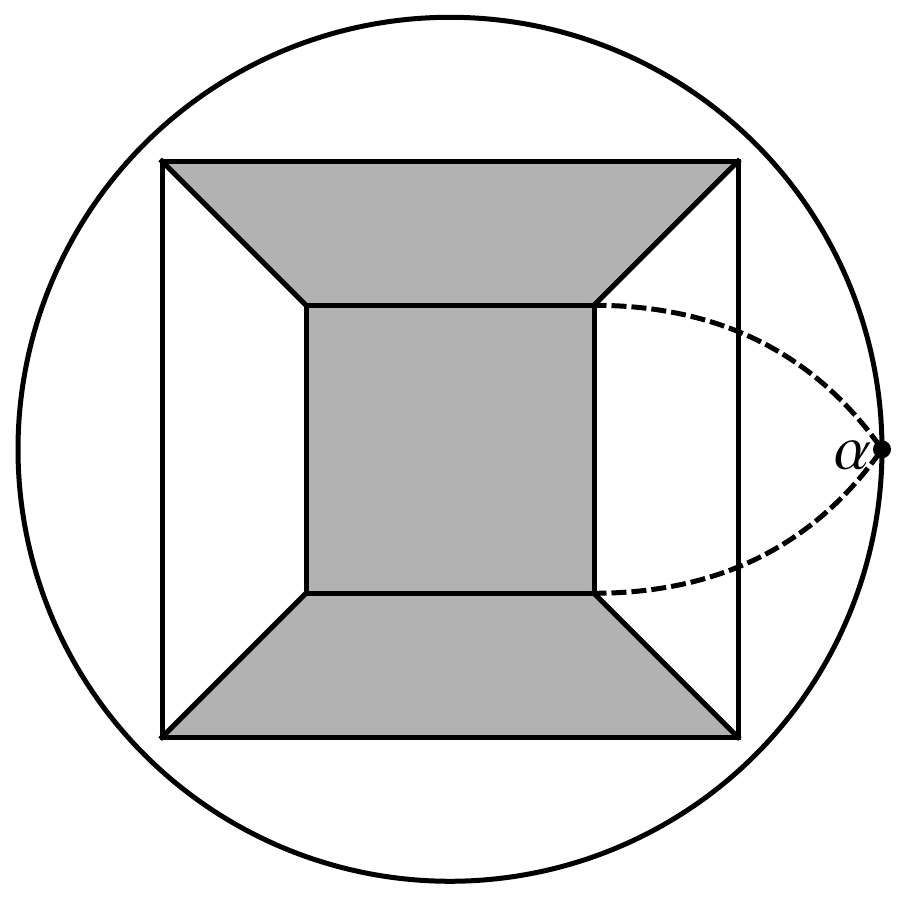}
}
\subfloat[$\QC_\beta$]{
\includegraphics[scale=.4]{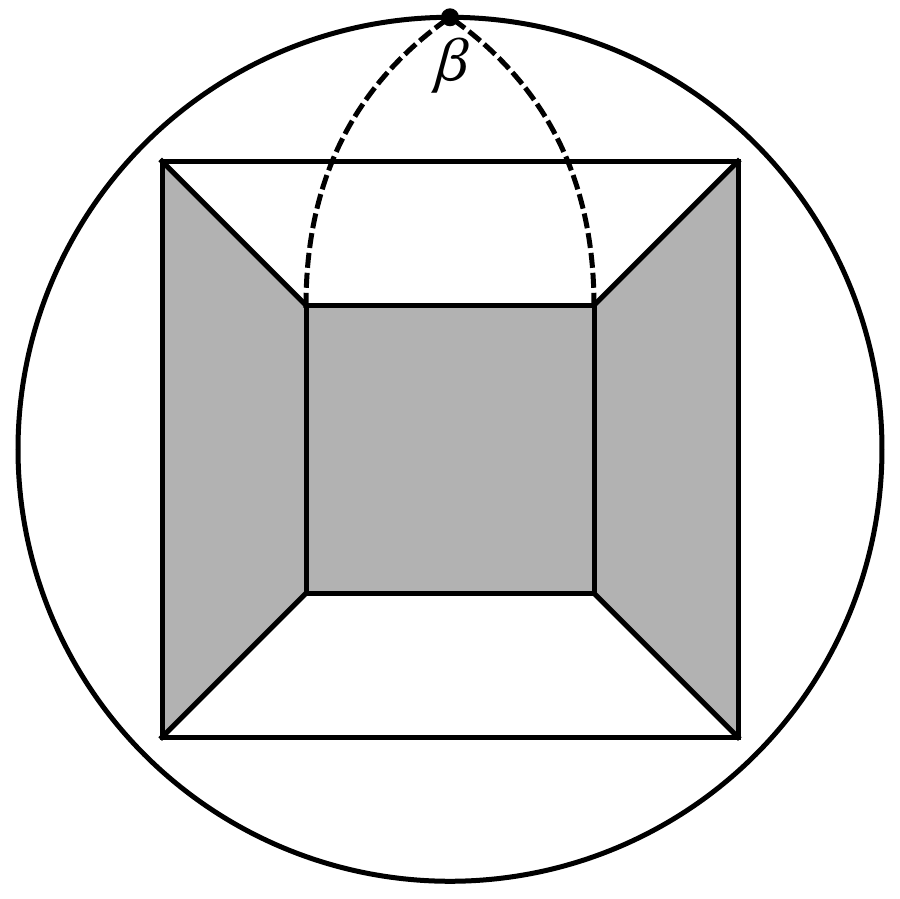}
}
\caption{Projective Lattice Complexes}\label{projective}
\end{figure}

\subsection{The simplicial case}  We close this section by showing that the complexes $\PC_W$ reduce to unions of stars of faces when $\PC=\Delta$ is a pure $n$-dimensional hereditary simplicial complex, as Proposition~\ref{star} shows.  We use the following lemma.

\begin{lem}\label{simplex}
Let $\Delta\subset\R^n$ be an $n+1$-simplex and $\sigma_1,\sigma_2\in\Delta$.  Then $\mbox{aff}(\sigma_1)\cap\mbox{aff}(\sigma_2)=\mbox{aff}(\sigma_1\cap\sigma_2)$.  This includes the case $\mbox{aff}(\sigma_1)\cap\mbox{aff}(\sigma_2)=\emptyset$, assuming $\mbox{aff}(\emptyset)=\emptyset$.
\end{lem}
\begin{proof} $\Delta$ is the convex hull of $n+1$ vertices $\{v_0,\ldots,v_n\}$.  Let $\Delta_i$ be the convex hull of $\{v_0,\ldots,v_{i-1},v_{i+1},\ldots,v_n\}$, the codimension one face of $\Delta$ determined by leaving out vertex $v_i$.  It has supporting hyperplane $H_{v_i}$, the affine hull of all vertices except $v_i$.  A $k$-dimensional face $\tau$ of $\Delta$ determined by $\{v_{i_0},\ldots,v_{i_k}\}$ has $\mbox{aff}(\tau)=\cap_{v\in\Delta_0\setminus\tau_0} H_v$.  Now if $\sigma_1,\sigma_2\in\Delta$, 
\[
\mbox{aff}(\sigma_1)\cap\mbox{aff}(\sigma_2)=\bigcap_{v\in (\Delta_0\setminus(\sigma_1)_0)\cup (\Delta_0\setminus(\sigma_2)_0)} H_v =\bigcap_{v\in\Delta_0\setminus(\sigma_1\cap\sigma_2)_0} H_v=\mbox{aff}(\sigma_1\cap\sigma_2)
\]
\end{proof}

\begin{lem}\label{starcomp}
Let $\Delta\subset\R^n$ be a pure $n$-dimensional hereditary simplicial complex, and $W\in\A(\Delta)$.  Then each component of $\Delta_W$ has the form $\mbox{st}(\tau)$ for some $\tau\in\Delta$.
\end{lem}
\begin{proof}
Let $G^i_W(\Delta)$ be a connected component of $G_W(\Delta)$ and $\Delta^i_W\subset\Delta$ the corresponding complex.  Let $\Delta^i_W$ have facets $\sigma_1,\ldots,\sigma_k$, set $V_i=\bigcap_{\tau\in(\sigma_i)_{n-1}\cap(\Delta^i_W)^0_{n-1}} \mbox{aff}(\tau)$, $V=\cap_{\tau\in(\Delta^i_W)^0_{n-1}} \mbox{aff}(\tau)=\cap_{i=1}^k V_i$.  By applying Lemma~\ref{simplex} iteratively, $V_i=\mbox{aff}(\gamma_i)$ for some face $\gamma_i\in\sigma_i$.  Now let $K$ be a walk of length $m+1$ on the graph $G^i_W$ with the corresponding sequence of facets and codimension $1$ faces of $\Delta^i_W$ being $S=\{\sigma_{j_0},\tau_{j_1},\sigma_{j_1},\ldots,\tau_{j_m},\sigma_{j_m}\}$, where $\tau_{j_i}=\sigma_{j_{i-1}}\cap\sigma_{j_i}$ is a codimension $1$ face of $\sigma_{j_{i-1}}$ and $\sigma_{j_i}$.  Set $\beta_0=\gamma_{j_0}$, $\beta_c=\cap_{i=0}^c \gamma_{j_i}$.  We prove $\cap_{i=0}^c V_{j_i}=\mbox{aff}(\beta_c)$, for $c=0,\ldots,m$ by induction.  We already have $V_{j_0}=\mbox{aff}(\gamma_{j_0})=\mbox{aff}(\beta_0)$.  Assume $\cap_{i=0}^c V_{j_i}=\mbox{aff}(\beta_c)$.  $\beta_c$ is a face of $\gamma_{j_c}$, which in turn is a face of $\tau_{j_{c+1}}$, since $\tau_{j_{c+1}}$ is a codimension $1$ face of $\sigma_{j_c}$ that contains $W$.  So $\beta_c$ is a face of $\sigma_{j_{c+1}}$.  By Lemma~\ref{simplex}, $\mbox{aff}(\beta_c)\cap\mbox{aff}(\gamma_{j_{c+1}})=\mbox{aff}(\beta_c\cap\gamma_{j_{c+1}})=\mbox{aff}(\beta_{c+1})$.  Putting everything together, we have $\cap_{i=0}^{c+1} V_{j_i}=(\cap_{i=0}^c V_{j_i})\cap V_{j_{c+1}}=\mbox{aff}(\beta_c)\cap\mbox{aff}(\gamma_{c+1})=\mbox{aff}(\beta_{c+1})$.

Setting $\tau=\beta_m$ and noting that $V=\cap_{i=1}^k V_i=\cap_{i=0}^m V_{j_i}$, we have $V=\mbox{aff}(\tau)$.  By construction $\tau$ is a face of $\sigma_i$ for $i=1,\ldots,k$, hence $\Delta^i_W\subset\mbox{st}(\tau)$.  On the other hand, since $W\subset\mbox{aff}(\tau)$, $\mbox{st}(\tau)\subset\Delta_W$.  $\Delta$ is hereditary, so $G(\mbox{st}(\tau))$ is connected and $\mbox{st}(\tau)$ is a component of $\Delta_W$.  Hence $\Delta^i_W=\mbox{st}(\tau)$.
\end{proof}

We indicate precisely which stars appear in $\Delta_W$, following notation of Billera and Rose \cite{Modules}.

\begin{prop}\label{star}
Let $\Delta$ be a pure simplicial complex and $W\in L_\Delta$.  Set
\[
S(W)=\{\tau\in\Delta|W \subset \mbox{aff}(\tau) \mbox{ and $\tau$ is minimal with respect to this property} \}.
\]
Then
\[
\Delta_W=\bigsqcup_{\tau\in S(W)} \mbox{st}(\tau)
\]
\end{prop}

\begin{proof}
Let $\Delta^i_W$ be a component of $\Delta_W$.  By Lemma~\ref{starcomp}, $\Delta^i_W=\mbox{st}(\tau)$ for some $\tau\in\Delta$.  Clearly $W\subset\mbox{aff}(\tau)$.  If $\tau$ is not minimal with respect to this property, then there is some $\gamma$ a proper face of $\tau$ so that $W\subset\mbox{aff}(\gamma)$.  But then $\mbox{st}(\tau)\subsetneq\mbox{st}(\gamma)\subset \Delta_W$, contradicting that $\mbox{st}(\tau)$ is a component of $\Delta_W$.  So $\tau\in S(P)$.  Now suppose $\tau\in S(P)$.  Clearly $\mbox{st}(\tau)\subset \Delta_W$ and $G(\mbox{st}(\tau))$ is connected since $\Delta$ is hereditary.  Hence $\mbox{st}(\tau)$ is contained in a component $\Delta^i_W$ of $\Delta_W$.  By Lemma~\ref{starcomp}, $\Delta^i_W=\mbox{st}(\gamma)$ for some $\gamma\in\Delta$.  We may assume $\tau$ is the intersection of the simplices contained in $\mbox{st}(\tau)$, implying $\gamma$ is a face of $\tau$.  But $W\subset\mbox{aff}(\gamma)$ and $\tau\in S(W)$, so $\tau=\gamma$.
\end{proof}

\section{Localization of $C^r(\PC)$}
Our objective in this section is to give an explicit description of $C^r(\PC)_P$ for any prime $P\subset R$, using the complexes $\PC_W$ defined in the previous section.  We first recall the definition of $C^r(\PC)$.  For $U\subset \R^n$, let $C^r(U)$ denote the set of functions $F:U\rightarrow \R$ continuously differentiable of order $r$.  For $F:|\PC|\rightarrow \R$ a function and $\sigma\in\PC$, $F_\sigma$ denotes the restriction of $F$ to $\sigma$.  The module $C^r(\PC)$ of piecewise polynomials continuously differentiable of order $r$ on $\PC$ is defined by
\[
C^r(\PC):=\{F\in C^r(|\PC|)|F_\sigma\in R \mbox{ for every } \sigma\in\PC_n\}
\]
The polynomial ring $R$ includes into $C^r(\PC)$ as globally polynomial functions (these are the \textit{trivial} splines); this makes $C^r(\PC)$ an 
$R$-algebra via pointwise multiplication.

Given a pure $n$-dimensional polytopal complex $\PC$, the boundary complex of $\PC$ is a pure $(n-1)$-dimensional complex denoted by $\partial\PC$.  For a pure $n$-dimensional subcomplex $\QC\subset\PC$, not necessarily hereditary, we use the following notation.
\begin{enumerate}
\item $C^r_\QC(\PC):=\{F\in C^r(\PC)|F_\sigma=0 \mbox{ for all }\sigma\in\PC_n \setminus \QC_n \}$.
\item $L_{\partial\QC}:=\prod_{\tau\in (\partial\QC)_{n-1} \setminus (\partial\PC)_{n-1}} l_\tau$.
\end{enumerate}

We can naturally view $C^r(\QC)$ as a submodule of $\oplus_{\sigma\in\PC_n} R$ by defining $F_\sigma=0$ for $F\in C^r(\QC),\sigma\in\PC_n\setminus\QC_n$.  In this way $C^r(\QC)$ and $C^r(\PC)$ are submodules of the same ambient module.  We will assume this throughout the paper.

$C^r_\QC(\PC)$ satisfies $L^{r+1}_{\partial \QC}\cdot C^r(\QC) \subseteq C^r_\QC(\PC)$.  This follows since $L^{r+1}_{\partial\QC}\cdot F$ vanishes on $\partial \QC\setminus \partial \PC$ to order $r+1$ for any $F\in C^r(\QC)$.  From this we get the following observation on localization which we refer to as $(\star)$.
\begin{flushleft}
\textbf{Observation}:  $C^r_\QC(\PC)_P=C^r(\QC)_P$ for any prime $P$ such that $L_{\partial\QC}\not\in P$. \qquad $(\star)$
\end{flushleft}
Let $\QC_1,\ldots,\QC_k$ be pure $n$-dimensional polyhedral subcomplexes of $\PC$.  Call $\QC_1,\ldots,\QC_k$ a \textbf{partition} of $\PC$ if the facets of $\QC_1,\ldots,\QC_k$ partition the facets of $\PC$.

\begin{lem}\label{2part}
Let $\QC$ and $\OC$ be two polyhedral subcomplexes which partition $\PC$.  Let $P$ be a prime of $R$ such that $L_{\partial\QC}\not\in P$ and $L_{\partial\OC}\not\in P$.  Then
\[
C^r(\PC)_P=C^r(\QC)_P \oplus C^r(\OC)_P
\]
as submodules of $\oplus_{\sigma\in\PC}R_P$.  More precisely, 
\[
(C^r_\QC(\PC)+C^r_\OC(\PC))_P=C^r(\QC)_P + C^r(\OC)_P.
\]
\end{lem}

\begin{proof}
It is clear that $C^r(\QC)\cap C^r(\OC)=0$ and $C^r_\QC(\PC)\cap C^r_\OC(\PC)=0$ in $\oplus_{\sigma\in\PC}R$ since $\QC$ and $\OC$ have no common facets.  So both $C^r_\QC(\PC)+C^r_\OC(\PC)$ and $C^r(\QC) + C^r(\OC)$ are internal direct sums.  The result follows from $(\star)$.
\end{proof}

\begin{cor}\label{mpart}
Let $\QC_1,\ldots,\QC_k$ be a partition of $\PC$ into polyhedral subcomplexes and $P\subset R$ a prime such that $L_{\partial\QC_i}\not\in P$ for $i=1,\ldots,k$.
\[
C^r(\PC)_P=\bigoplus_{i=1}^k C^r(\QC_i)_P
\]
as submodules of $\oplus_{\sigma\in\PC_n}R$.
\end{cor}

\begin{proof}
Apply Corollary~\ref{2part} iteratively.
\end{proof}

\begin{cor}\label{easyP}
Let $P$ be a prime of $R$ so that $l_\tau\not\in P$ for every edge $\tau\in\PC_1^0$.  Then
\[
C^r(\PC)_P=\bigoplus_{\sigma\in\PC_n}R_P
\]
\end{cor}

\begin{proof}
Apply Corollary~\ref{mpart} to the partition of $\PC$ into individual facets.  This yields
\[
C^r(\PC)_P=\bigoplus_{\sigma\in\PC_n} C^r(\sigma)_P
\]
Since $C^r(\sigma)=R$, we are done.
\end{proof}

Now let $I\subset R$ be an ideal and $\PC_I\subset\PC$ be the subcomplex defined in the previous section.

\begin{defn}
Let $\PC_I$ have components $\PC^1_I,\ldots, \PC^k_I$.  Define
\[
C^r(\PC_I):=\bigoplus_{i=1}^k C^r(\PC^i_I).
\]
\end{defn}

\begin{prop}\label{local}
Let $P\subset R$ be a prime ideal, and $\PC\subset\R^n$ a polyhedral complex.  Then
\[
\begin{array}{rl}
C^r(\PC)_P & =C^r(\PC_P)_P \oplus \bigoplus\limits_{\sigma\in\PC_n\setminus(\PC_P)_n} R_P\\
 & =C^r(\PC_W)_P \oplus \bigoplus\limits_{\sigma\in\PC_n\setminus(\PC_W)_n} R_P
\end{array},
\]
where $W$ is the unique flat of $L_\PC$ so that $\PC_W=\PC_P$ guaranteed by Lemma~\ref{lat}.
\end{prop}

\begin{proof}
Consider the partition of $\PC$ determined by $\PC_P$ and $\QC$, where $\QC$ is the subcomplex of $\PC$ generated by all facets $\sigma\in\PC_n\setminus(\PC_P)_n$.  By the construction of $\PC_P$, $\tau\in(\PC_P)_{n-1}^0\iff l_\tau\in P$.  Hence if $\tau\in\partial\PC_P\setminus\partial\PC=\partial\QC\setminus\partial\PC$ then $l_\tau\not\in P$.  Applying Corollary~\ref{2part} we have
\[
C^r(\PC)_P=C^r(\PC_P)_P \oplus C^r(\QC)_P
\]
Again since all $\tau\in\PC^0_{n-1}$ such that $l_\tau\in P$ are interior codim $1$ faces of $\PC_P$, there is no $\tau\in\QC^0_{n-1}$ such that $l_\tau\in P$.  Applying Corollary~\ref{easyP} to $C^r(\QC)_P$ gives the result.
\end{proof}

We get the following result of Billera and Rose, used in the proof of Theorem 2.3 in \cite{Modules}, as a corollary of Proposition~\ref{local} and Proposition~\ref{star}.

\begin{cor}
Let $\Delta$ be a simplicial complex and $W\in L_\Delta$.  Define $S(W)$ as in Proposition~\ref{star}. Then
\[
C^r(\Delta)_P=\bigoplus_{\tau\in S(W)} C^r(\mbox{st}(\tau))_P,
\]
where $W\in L_\Delta$ is the unique flat so that $\Delta_P=\Delta_W$.
\end{cor}

Note that if a facet $\sigma$ is in $S(W)$ then it is a not a facet of $\Delta_W$. Since $C^r(\sigma)=R$, the sum $\bigoplus_{\sigma\in\Delta_n\setminus(\Delta_W)_n} R_P$ appearing in Theorem~\ref{local} is implicit in the sum above.

\subsection{Relation to results of Billera-Rose and Yuzvinsky}
As an application of Proposition~\ref{local}, we prove a slight variant of a 
result of Yuzvinsky \cite{Yuz} which reduces computation of the projective dimension of $C^r(\PC)$ to the central case.  Recall that if $M$ is a module over the polynomial ring $R$, a \textit{finite free resolution} of $M$ of length $r$ is an exact sequence of free modules
\[
F_\bullet: \mbox{ }0\rightarrow F_r\xrightarrow{\phi_r} F_{r-1}\xrightarrow{\phi_{r-1}} \cdots \xrightarrow{\phi_1} F_0
\]
such that $\mbox{coker }\phi_1=M$.  The Hilbert syzygy theorem guarantees that $M$ has a finite free resolution.  The \textit{projective dimension} of $M$, denoted $\mbox{pd}(M)$, is the minimum length of a finite free resolution.  If $M$ is a graded $S$-module with $\mbox{pd}(M)=\delta$ then $M$ has a minimal free resolution $F_\bullet\rightarrow M$ of length $\delta$, unique up to graded isomorphism.  This resolution is characterized by the property that the entries of any matrix representing the differentials $\phi.$ in $F_\bullet$ are contained in the homogeneous maximal ideal $(x_1,\ldots,x_n)$.

In \cite{Yuz}, Yuzvinsky introduces a poset $L$, different from $L_\PC$, associated to a polytopal complex.  He defines subcomplexes associated to each $x\in L$ and uses them to reduce the characterization of projective dimension to the graded case (Proposition 2.4).  Proposition~\ref{local} is the analog for $L_\PC$  of Lemma 2.3 in \cite{Yuz}, and we use it to prove the following statement.

\newpage

\begin{thm}\label{projcit}Let $\PC\subset\R^n$ be a polytopal complex. Then
\begin{enumerate}
\item $\mbox{pd}(C^r(\PC))\geq \mbox{pd}(C^r(\PC_W))$ for all $W\in L_\PC$.
\item $\mbox{pd}(C^r(\PC))=\max\limits_{W\in L_\PC} \mbox{pd}(C^r(\PC_W))$.  In particular, $C^r(\PC)$ is free iff $C^r(\PC_W)$ is free for all nonempty $W\in L_\PC$.
\end{enumerate}
\end{thm}
\begin{proof}
We use the following two facts about projective dimension. Here $M$ is any $R$-module, not necessarily graded.
\begin{enumerate}
\item[(A)] For any prime $P\subset R$, $\mbox{pd}(M)\geq \mbox{pd}(M_P)$.
\item[(B)] $\mbox{pd}(M)=\max\limits_{P\in\mbox{Spec }R} \mbox{pd}(M_P)$.
\end{enumerate}
(A) is immediate from localizing any free resolution $F_\bullet$ of $M$.  (B) follows from showing that there are primes which preserve $\mbox{pd}(M)$ under localization.  Set $\mbox{pd}(M)=r$.  We have $\mbox{Ext}^r_R(M,R)\neq 0$ and taking any prime $P$ in its support will suffice.  For such a prime, we have $\mbox{Ext}^r_{R_P}(M_P,R_P)\cong \mbox{Ext}^r_R(M,R)\otimes_R R_P \neq 0$.  Hence $\mbox{pd}(M_P)\geq r$.  Since we already have $\mbox{pd}(M_P)\leq r$, $\mbox{pd}(M_P)=r$.
(1) Observe that $C^r(\PC_W)$ is graded with respect to $I(\xi)$ for any point $\xi\in W$ since the affine span of every interior codimension $1$ face of $\PC_W$ contains $W$, hence also contains $\xi$.  Choose $\xi\in W\setminus\cup_{V\subset W} V$, where $V\in L_\PC$.  Then $\PC_\xi=\PC_W$.  We have
\[
\mbox{pd}(C^r(\PC))\geq \mbox{pd}(C^r(\PC)_{I(\xi)})=\mbox{pd}(C^r(\PC_\xi)_{I(\xi)})=\mbox{pd}(C^r(\PC_W)_{I(\xi)}),
\]
where the first equality follows from fact (A) above and the second follows from Theorem~\ref{local} since $C^r(\PC)_{I(\xi)}$ is the direct sum of a free module and $C^r(\PC_\xi)_{I(\xi)}$.  $C^r(\PC_W)$ is graded with respect to $I(\xi)$, so a minimal resolution of $C^r(\PC_W)$ has differentials with entries in $I(\xi)$.  It follows that this remains a minimal resolution under localization with respect to $I(\xi)$, so $\mbox{pd}(C^r(\PC_W)_{I(\xi)})=\mbox{pd}(C^r(\PC_W))$, and the result follows.
(2) Set $m=\max\limits_{W\in L_\PC} \mbox{pd}(C^r(\PC_W))$.  From (1) $\mbox{pd}(C^r(\PC))\geq m$.  For any prime $P \subset R$ we have $\PC_P=\PC_W$ for some $W\in L_\PC$ by Lemma~\ref{lat}, so
\[
\mbox{pd}(C^r(\PC)_P)=\mbox{pd}(C^r(\PC_W)_P)\leq \mbox{pd}(C^r(\PC_W)) \leq m.
\]
Hence $\mbox{pd}(C^r(\PC))\leq m$ from fact (B) above, and $\mbox{pd}(C^r(\PC))=m$ as desired.
\end{proof}

Since the $\PC_W$ are central complexes, this reduces computation of $\mbox{pd}(C^r(\PC))$ to the central case.  Via Proposition~\ref{star} we obtain the following theorem of Billera and Rose as a corollary to Theorem~\ref{projcit}.  Recall an $R$-module $M$ is free iff $\mbox{pd}(M)=0$.

\begin{thm}\label{old}[2.3 of \cite{Modules}]
Let $\PC\subset\R^n$ be a polytopal complex.  Then
\begin{enumerate}
\item If $C^r(\PC)$ is free over $R$ then $C^r(\mbox{st}(\tau))$ is free over $R$ for all $\tau\neq\emptyset\in\PC$.
\item If $\PC=\Delta$ is simplicial then the converse is also true: if $C^r(\mbox{st}(\sigma))$ is free for all nonempty $\sigma\in\Delta$, then $C^r(\Delta)$ is free.
\end{enumerate}
\end{thm}

\section{Lattice-Supported Splines}
Our main application of Proposition~\ref{local} is to construct ``locally supported approximations'' to $C^r(\PC)$ which allow us to generalize the notion of a star-supported basis of $C^r_d(\PC)$ \cite{NonEx} in the simplicial case.  In particular we show that a basis of $C^r_d(\PC)$ consisting of splines supported on complexes of the form $\PC_W$ always exists for $d\gg 0$, \textit{as long as we allow} $W\in L_{\widehat{\PC}}$.

Recall for $\QC$ a pure $n$-dimensional connected subcomplex of $\PC\subset\R^n$, we defined $C^r_\QC(\PC)$ to be the set of splines vanishing outside of $\QC$.  Also recall that $\widehat{\PC}\subset\R^{n+1}$ is the central complex obtained by coning over $\PC$ in $\R^{n+1}$.  In the remark in subsection~\ref{homog}, we defined subcomplexes $\PC_W\subset\PC$ for $W\in L_{\widehat{\PC}}$ by slicing the subcomplex $\widehat{\PC}_W$ with the hyperplane $x_0=1$.  We make the following definitions.

\begin{defn}
For $W\in L_{\widehat{\PC}}$ let $\PC_W$ have components $\PC^1_W,\ldots,\PC^m_W$.
\begin{enumerate}
\item $C^r_W(\PC):=\bigoplus\limits_{i=1}^m C^r_{\PC^i_W}(\PC)$
\item $LS^{r,k}(\PC):=\sum\limits_{\substack{W\in L_{\widehat{\PC}}\\ 0\leq\mbox{\emph{rk}}(W)\leq k}} C^r_W(\PC)$
\end{enumerate}
\end{defn}

$C^r_W(\PC)$ is the submodule of $C^r(\PC)$ generated by splines which are $0$ outside of a component of $\PC_W$.  We say that elements of $C^r_W(\PC)$ are \textit{supported at} $W$.  If a spline $F$ is supported at some $W\in L_{\widehat{\PC}}$ then we call $F$ \textit{lattice-supported}.  By Proposition~\ref{star}, \textit{lattice-supported} splines generalize the notion of \textit{star-supported} splines \cite{NonEx} to polytopal complexes.

\textbf{Remark:} Results in this section which do not depend on grading (Proposition~\ref{krank}, Theorem~\ref{latticegens}, and Corollary~\ref{ungeq}) only require the sum in $(2)$ to run over $W\in L_\PC$, not $W\in L_{\widehat{\PC}}$.

\begin{prop}\label{krank}
Let $\PC\subset\R^n$ be a pure $n$-dimensional polytopal complex.  If $P\subset R$ is a prime such that the unique $W\in L_\PC$ satisfying $\PC_P=\PC_W$, guaranteed by Lemma~\ref{lat}, has $\mbox{\emph{rk}}(W)\leq k$, then $LS^{r,k}(\PC)_P=C^r(\PC)_P$.
\end{prop}

\begin{proof}
Let $P$ satisfy the given condition.  Since $\mbox{rk}(W)\leq k$ the module $C^r_W(\PC)$ appears as a summand in $LS^{r,k}(\PC)$.  Note that $C^r_{\R^n}(\PC)=\sum_{\sigma\in\PC_n}C^r_\sigma(\PC)$.  Define the submodule $N(P)\subset LS^{r,k}(\PC)$ by $N(P)=C^r_W(\PC)+\sum_{\sigma\in\PC_n\setminus(\PC_P)_n}C^r_\sigma(\PC)$.  The support of the summands of $N(P)$ is disjoint, so it is clear that this is an internal direct sum.  Also, for $\sigma\not\in(\PC_P)_n$, $C^r_\sigma(\PC)_P=R_P$.  We have
\[
\begin{array}{rl}
N(P)_P & = C^r_W(\PC)_P \oplus \bigoplus\limits_{\sigma\in \PC_n\setminus(\PC_P)_n} C^r_\sigma(\PC)_P \\
 & =C^r(\PC_W)_P \oplus \bigoplus\limits_{\sigma\in\PC_n\setminus(\PC_P)_n} R_P\\
 &=C^r(\PC)_P
\end{array}
\]
where the second equality follows from applying observation $(\star)$ to the summands of $C^r_W(\PC)$ and the third equality follows from Theorem~\ref{local}.  Since $N(P)\subset LS^{r,k}(\PC) \subset C^r(\PC)$ and $N(P)_P=C^r(\PC)_P$, we have $LS^{r,k}(\PC)_P=C^r(\PC)_P$.
\end{proof}

\begin{thm}\label{latticegens}
Let $\PC\subset\R^n$ be a polytopal complex.  Then $LS^{r,k}(\PC)$ fits into a short exact sequence
\[
0\rightarrow LS^{r,k}(\PC) \rightarrow C^r(\PC) \rightarrow C \rightarrow 0
\]
where $C$ has codimension $\geq k+1$ and the primes in the support of $C$ with codimension $k+1$ are contained in the set $\{I(W)|W\in L_\PC\mbox{ \emph{and} }\mbox{\emph{rk}}(W)=k+1\}$.
\end{thm}
Recall the \textit{support} of $C$ is the set of primes $P\subset R$ satisfying $C_P\neq 0$.
\begin{proof}
The first statement follows from Proposition~\ref{krank}.  Now suppose $C$ is supported at $P$ with $\mbox{codim }P=k+1$, and let $W\in L_\PC$ be the unique flat so that $\PC_P=\PC_W$ (Lemma~\ref{lat}).  If $\mbox{rk}(W)\leq k$ then $LS^{r,k}(\PC)_P=C^r(\PC)_P$ by Proposition~\ref{krank} and $C_P=0$.  So we must have $\mbox{rk}(W)=k+1$ and $P=I(W)$.  Hence the primes of codimension $k+1$ in the support of $C$ are contained in $\{I(W)|W\in L_\PC\mbox{ and rk}(W)=k+1\}$.
\end{proof}

Setting $k=n$ in Theorem~\ref{latticegens}, we obtain

\begin{cor}\label{ungeq}
For $\PC\subset\R^n$, $LS^{r,n}(\PC)=C^r(\PC)$.
\end{cor}

\subsection{Graded Results}

Let $R_{\le d}$ and $R_d$ be the set of polynomials $f\in R=\R[x_1,\ldots,x_n]$ of degree $\le d$ and degree $d$, respectively.  For $\PC\subset\R^n$ we have a filtration of $C^r(\PC)$ by $\R$-vector spaces
\[
C^r_d(\PC):=\{F\in C^r(\PC)|F_{\sigma}\in R_{\le d} \mbox{ for all facets }\sigma\in\PC_n\}.
\]
If $\PC$ is a central complex, $C^r(\PC)$ is graded by the vector spaces
\[
C^r(\PC)_d:=\{F\in C^r(\PC)|F_{\sigma}\in R_d \mbox{ for all facets }\sigma\in\PC_n\}.
\]
In other words, $C^r(\PC)=\bigoplus_{d\ge 0} C^r_d(\PC)$.  We have the following lemma due to Billera and Rose.

\begin{lem}[Theorem 2.6 of \cite{DimSeries}]\label{splinehom}
$C^r_d(\PC)\cong C^r(\widehat{\PC})_d$ as $\R$-vector spaces.
\end{lem}

The map $\phi_d: C^r_d(\PC)\rightarrow C^r(\widehat{\PC})_d$ in Lemma~\ref{splinehom} is given by \textit{homogenizing}:
\[
f(x_1,\ldots,x_n)\in C^r_d(\PC) \rightarrow x_0^d f(x_1/x_0,\ldots,x_n/x_0)\in C^r(\widehat{\PC})_d,
\]
while its inverse $\phi_d^{-1}:C^r(\widehat{\PC})_d \rightarrow C^r_d(\PC)$ is given by
\[
F(x_0,\ldots,x_n)\in C^r(\widehat{\PC})_d \rightarrow F(1,x_1,\ldots,x_n)\in C^r_d(\PC).
\]
We define parallel filtrations and gradings for $LS^{r,k}(\PC)$.  For $W\in L_{\widehat{\PC}}$, define
\[
\begin{array}{rl}
C^r_{W,d}(\PC):=&C^r_W(\PC)\cap C^r_d(\PC)\\
C^r_W(\widehat{\PC})_d:=&C^r_W(\widehat{\PC})\cap C^r(\widehat{\PC})_d
\end{array}
\]
Then $LS^{r,k}(\PC)$ has a filtration by vector spaces
\[
LS^{r,k}_d(\PC):=\sum\limits_{\substack{W\in L_{\widehat{\PC}}\\ 0\leq\mbox{rk}(W)\leq k}} C^r_{W,d}(\PC),
\]
which are subspaces of $C^r_d(\PC)$.  If $\PC$ is a central complex, $LS^{r,k}(\PC)$ is graded by the vector spaces
\[
\begin{array}{rl}
LS^{r,k}(\PC)_d:= & \{F\in LS^{r,k}(\PC)|F_{\sigma}\in R_d \mbox{ for all facets }\sigma\in\PC_n\}\\[5 pt]
= &\sum\limits_{\substack{W\in L_{\widehat{\PC}}\\ 0\leq\mbox{rk}(W)\leq k}} C^r_W(\widehat{\PC})_d,
\end{array}
\]
which are subspaces of $C^r(\widehat{\PC})_d$.  If $\PC$ is central, the set of subcomplexes $\PC_W$ for $W\in L_{\widehat{\PC}}$ is the same as the set of subcomplexes $\PC_W$ for $W\in L_\PC$ (there is a rank preserving isomorphism between $L_{\PC}$ and $L_{\widehat{\PC}}$ in this case).  Hence we may write
\[
LS^{r,k}(\PC)_d=\sum\limits_{\substack{W\in L_\PC \\ 0\leq\mbox{rk}(W)\leq k}} C^r_W(\widehat{\PC})_d
\]

\begin{lem}\label{latticehom}
$LS^{r,k}_d(\PC)$ and $LS^{r,k}(\widehat{\PC})_d$ are isomorphic as $\R$-vector spaces.
\end{lem}

\begin{proof}
We show that the homogenization map $\phi_d:C^r_d(\PC)\rightarrow C^r(\widehat{\PC})_d$ restricts to an isomorphism between $LS^{r,k}_d(\PC)$ and $LS^{r,k}(\widehat{\PC})_d$. We have
\[
\begin{array}{rl}
LS^{r,k}_d(\PC):= & \sum\limits_{\substack{W\in L_{\widehat{\PC}}\\ 0\leq\mbox{rk}(W)\leq k}} C^r_{W,d}(\PC)\\
LS^{r,k}(\widehat{\PC})_d= & \sum\limits_{\substack{W\in L_{\widehat{\PC}}\\ 0\leq\mbox{rk}(W)\leq k}} C^r_W(\widehat{\PC})_d
\end{array}
\]
Since $\phi_d$ is $\R$-linear, it suffices to show that, given $W\in L_{\widehat{\PC}}$, $\phi_d$ restricts to an isomorphism $\phi_d:C^r_{W,d}(\PC)\rightarrow C^r_W(\widehat{\PC})_d$.  We show $\phi_d(C^r_{W,d}(\PC))\subset C^r_W(\widehat{\PC})_d$ and $\phi^{-1}_d(C^r_W(\widehat{\PC})_d)\subset C^r_{W,d}(\PC)$.  Suppose $f\in C^r_{W,d}(\PC)$.  The support of $f$ is by definition contained in the subcomplex $\PC_W$, so the support of $\phi_h(f)$ is contained in the cone over $\PC_W$, which is precisely $\widehat{\PC}_W$.  It follows that $\phi_d(f)\in C^r_W(\widehat{\PC})$.  Since $\phi_d(f)\in C^r(\widehat{\PC})_d$, $\phi_d(f)\in (C^r_W(\widehat{\PC})\cap C^r(\widehat{\PC})_d)=C^r_W(\widehat{\PC})_d$.  Now suppose $F\in C^r_W(\widehat{\PC})_d$.  The support of $F$ is contained in the subcomplex $\widehat{\PC}_W$, so the support of $\phi^{-1}_h(F)$ is contained in the complex obtained by slicing $\widehat{\PC}_W$ with the hyperplane $x_0=1$.  But this is by definition $\PC_W$, so $\phi^{-1}_d(F)\in C^r_W(\PC)$.  Since $\phi^{-1}_d(F)\in C^r_d(\PC)$, $\phi^{-1}_d(F)\in (C^r_W(\PC)\cap C^r_d(\PC)) =C^r_{W,d}(\PC)$.
\end{proof}

For the following theorem, recall that the \textit{Hilbert function} $HF(M,d)$ of a finitely generated graded module $M=\bigoplus_{d} M_d$ over $R=\R[x_1,\ldots,x_n]$ is defined by
\[
HF(M,d)=\mbox{dim}_\R M_d
\]
and the \textit{Hilbert polynomial} $HP(M,d)$ of $M$ is the polynomial with which $HF(M,d)$ agrees for $d\gg 0$.

\begin{thm}\label{main}
Let $\PC\subset\R^n$ be a polytopal complex.
\begin{enumerate}
\item If $\PC$ is central, the first $k+1$ coefficients of $HP(C^r(\PC))$ and $HP(LS^{r,k}(\PC))$ agree.  In particular, $LS^{r,n-1}(\PC)_d=C^r(\PC)_d$ for $d\gg 0$.
\item $LS^{r,n}_d(\PC)=C^r_d(\PC)$ for $d\gg 0$.  Equivalently, $C^r_d(\PC)$ has a basis consisting of lattice-supported splines for $d\gg 0$.
\end{enumerate}
\end{thm}

\begin{proof}
(1) Applying Theorem~\ref{latticegens} to $LS^{r,k}(\PC)$ yields the short exact sequence
\[
0\rightarrow LS^{r,k}(\PC) \rightarrow C^r(\PC) \rightarrow C \rightarrow 0,
\]
where $\mbox{codim } C\le k+1$.  It follows that $HP(C,d)$ has degree at most $(n-1)-(k+1)$.  On the other hand, $HP(C^r(\PC),d)$ has degree $n-1$.  Since $HP(C^r(\PC),d)-HP(LS^{r,k}(\PC),d)=HP(C,d)$, the first $k+1$ coefficents of $HP(C^r(\PC),d)$ and $HP(LS^{r,k}(\PC),d)$ agree.  Now specialize to $k=n-1$.  Then $HP(C,d)=0$ so $HP(C^r(\PC),d)=HP(LS^{r,n-1}(\PC),d)$, implying $HF(C^r(\PC),d)=HF(LS^{r,n-1}(\PC),d)$ for $d\gg 0$.  Since $LS^{r,n-1}(\PC)_d \subset C^r(\PC)_d$, we have $LS^{r,n-1}(\PC)_d = C^r(\PC)_d$ for $d\gg 0$.
(2) From part (1), $LS^{r,n}(\widehat{\PC})_d = C^r(\widehat{\PC})_d$ for $d\gg 0$.  The result follows by applying Lemma~\ref{latticehom} to the left hand side and Lemma~\ref{splinehom} to the right hand side.
\end{proof}

\begin{exm}\label{Graded}
We give an example to highlight the difference between Corollary~\ref{ungeq} and Theorem~\ref{main} part (2).  Take the underlying complex to be the complex $\QC$ from Figure~\ref{SEVL}.  In Figure~\ref{1decomp} we show a decomposition for the unit in $C^0(\QC)$, $\mathbf{1}=\sum_{j=1}^5 G_j$, with $G_j\in C^0_2(\QC)$.  The splines $G_1,\ldots,G_5$ have support in the subcomplexes $\QC_W$ for $W\in L_\QC$.  Given any spline $F\in C^0_d(\QC)$, $\sum_{j=1} G_j\cdot F$ gives a decomposition of $F$ using lattice-supported splines in $C^0_{d+2}(\QC)$.   It follows that $C^0(\QC)=\sum_{W\in L_\QC} C^0_W(\QC)$, \textit{without} using the two complexes $\QC_\alpha,\QC_\beta$ of Figure~\ref{projective}, which correspond to intersections `at infinity.'  This is true in general; the statement of Corollary~\ref{ungeq} can be changed to
\[
\sum\limits_{W\in L_\PC} C^r_W(\PC)=C^r(\PC),
\]
without altering the proof.

However, if we want to write every spline $F\in C^0_d(\QC)$ as a sum of lattice-supported splines of degree \textit{at most} $d$, we must use the two complexes $\QC_\alpha$ and $\QC_\beta$.  For example, a computation in Macaulay2 shows that the spline $x^2\cdot\mathbf{1}\in C^0_2(\QC)$ is not in the vector space $\sum_{W\in L_\PC} C^0_{W,2}(\QC)$, while a decomposition for $x^2\cdot\mathbf{1}$ in $\sum_{W\in L_{\widehat{\PC}}} C^0_{W,2}(\QC)$ is shown in Figure~\ref{x2decomp}.  Set $l_1=x+1, l_2=y-1, l_3=x-1, l_4=y+1, l_5=x-y, l_6=x+y$ below.
\begin{figure}[htp]
\begin{flushleft}
\begin{minipage}{.2\textwidth}
\centering
\includegraphics[scale=.25]{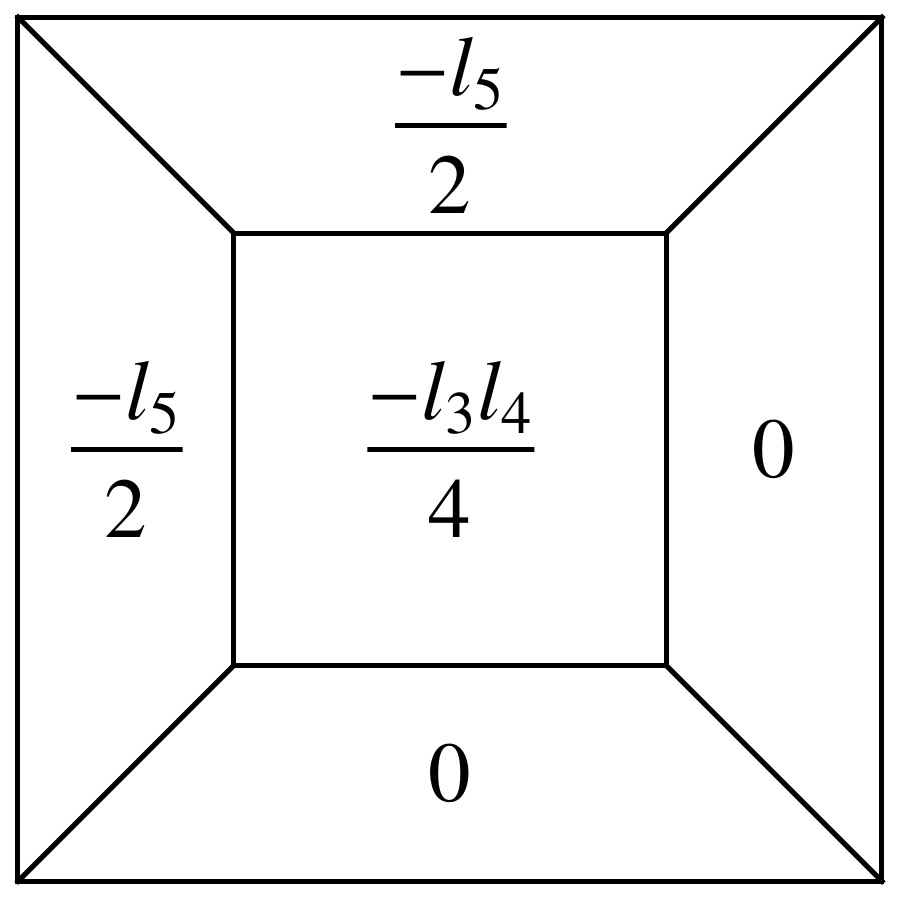}
\end{minipage}
\begin{minipage}{.02\textwidth}
\centering
$+$
\end{minipage}
\begin{minipage}{.2\textwidth}
\centering
\includegraphics[scale=.25]{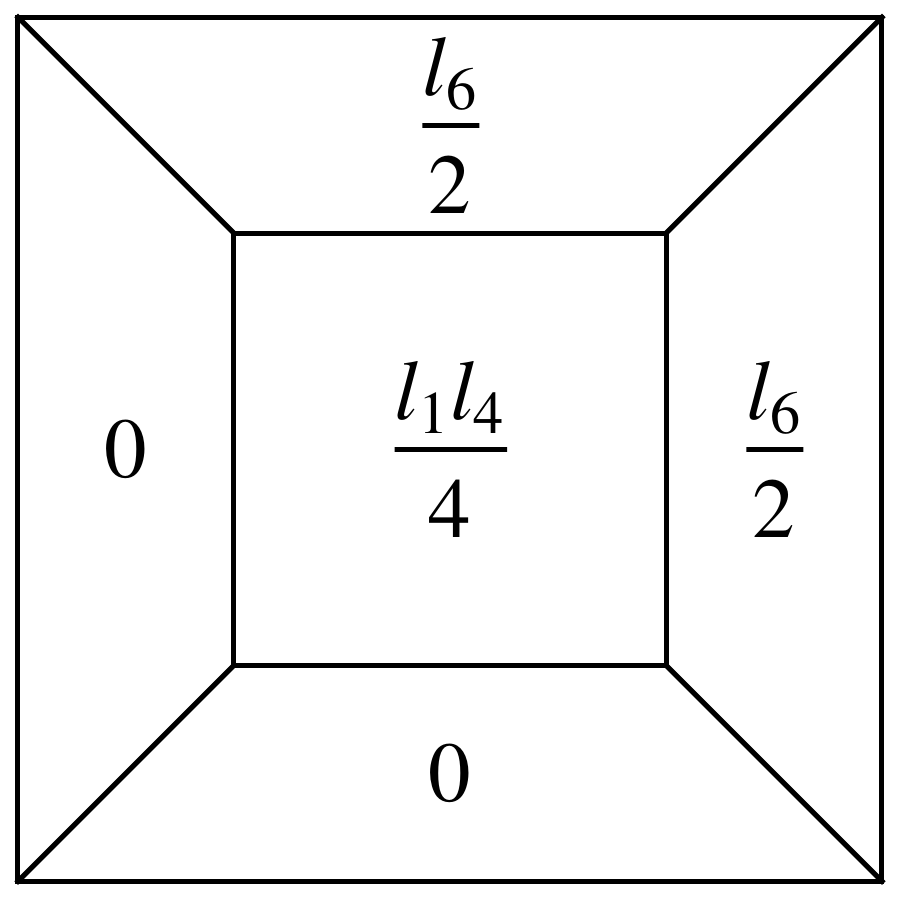}
\end{minipage}
\begin{minipage}{.02\textwidth}
\centering
$+$
\end{minipage}
\begin{minipage}{.2\textwidth}
\centering
\includegraphics[scale=.25]{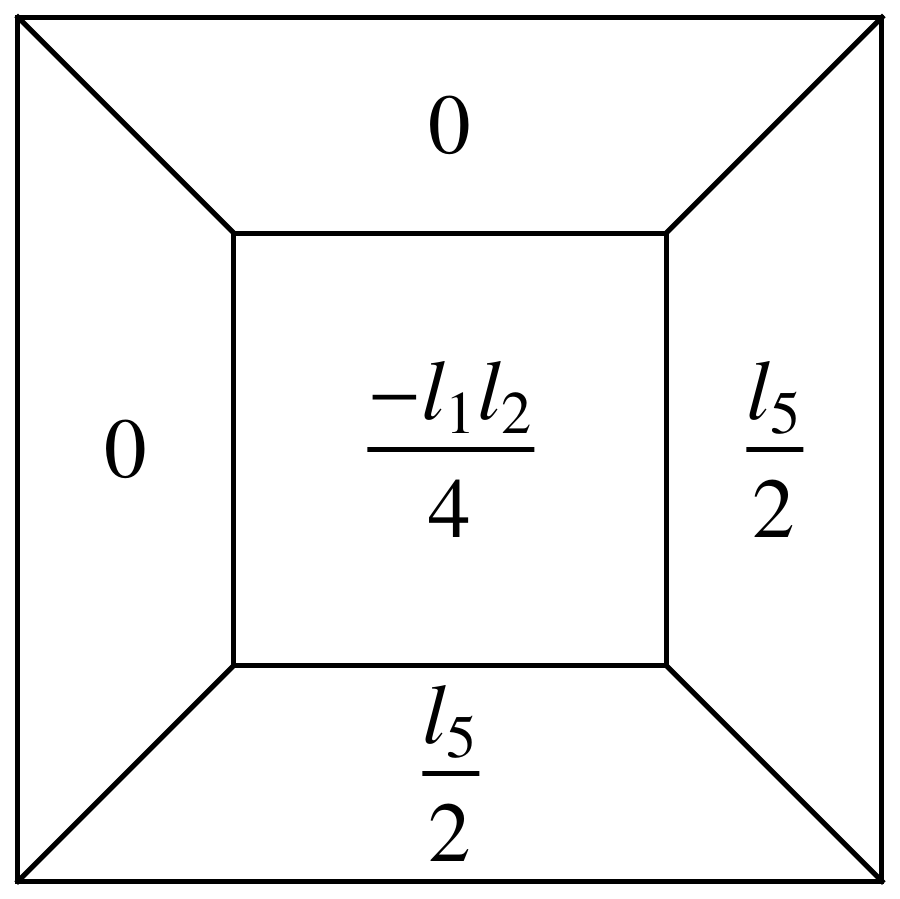}
\end{minipage}
\begin{minipage}{.02\textwidth}
\centering
$+$
\end{minipage}
\begin{minipage}{.2\textwidth}
\centering
\includegraphics[scale=.25]{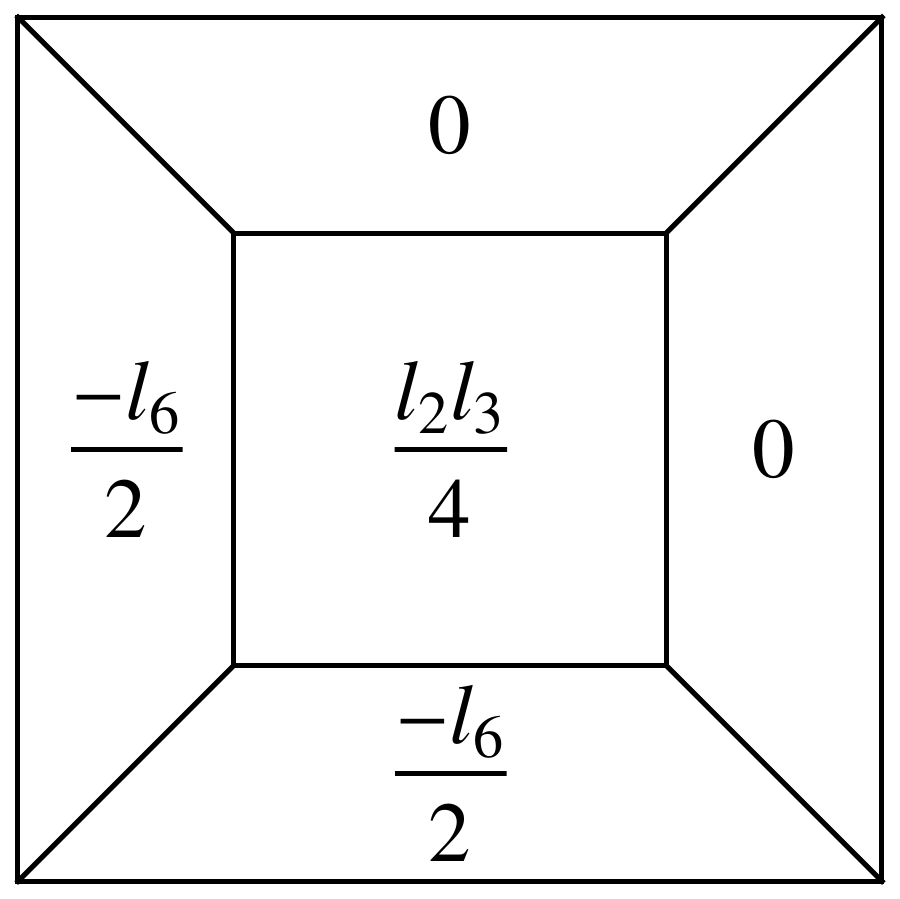}
\end{minipage}
\begin{minipage}{.02\textwidth}
\centering
$+$
\end{minipage}

\begin{minipage}{.2\textwidth}
\centering
\includegraphics[scale=.25]{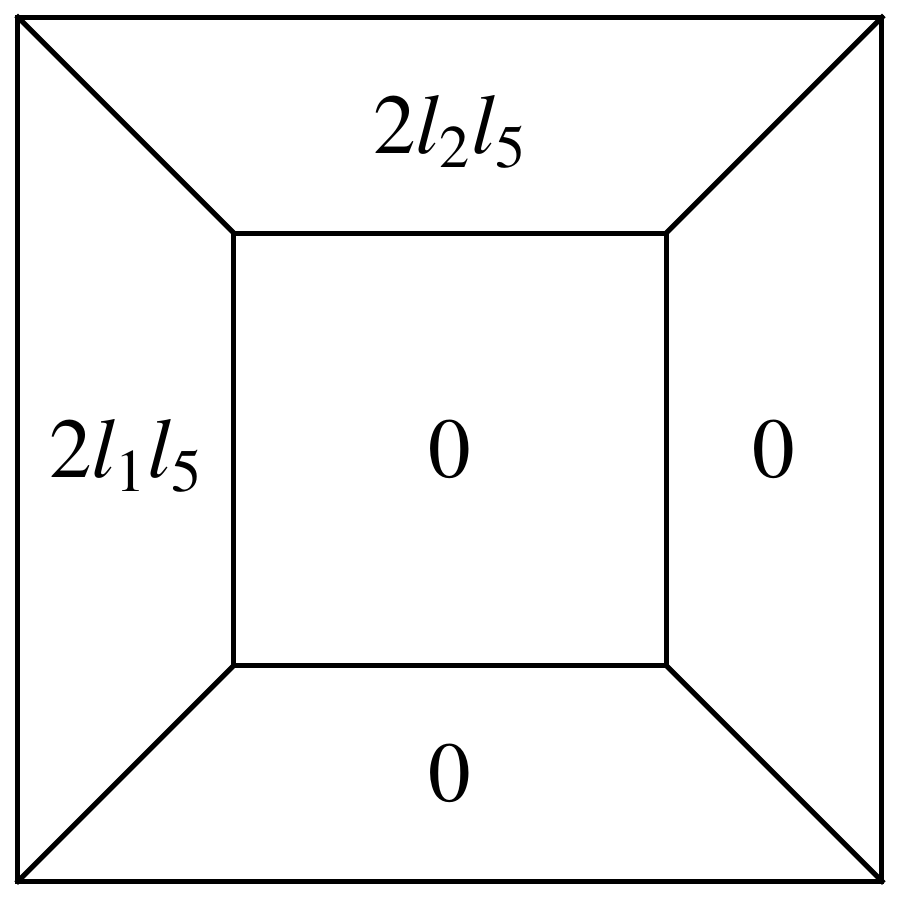}
\end{minipage}
\begin{minipage}{.02\textwidth}
\centering
$+$
\end{minipage}
\begin{minipage}{.2\textwidth}
\centering
\includegraphics[scale=.25]{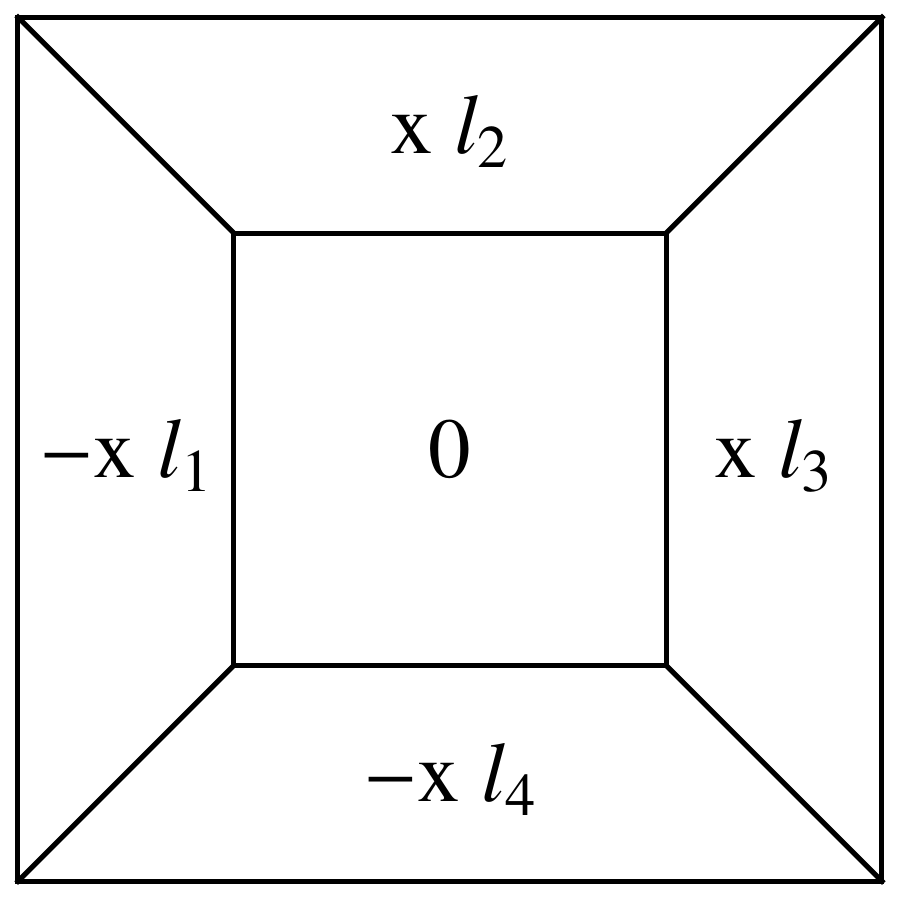}
\end{minipage}
\begin{minipage}{.02\textwidth}
\centering
$+$
\end{minipage}
\begin{minipage}{.2\textwidth}
\centering
\includegraphics[scale=.25]{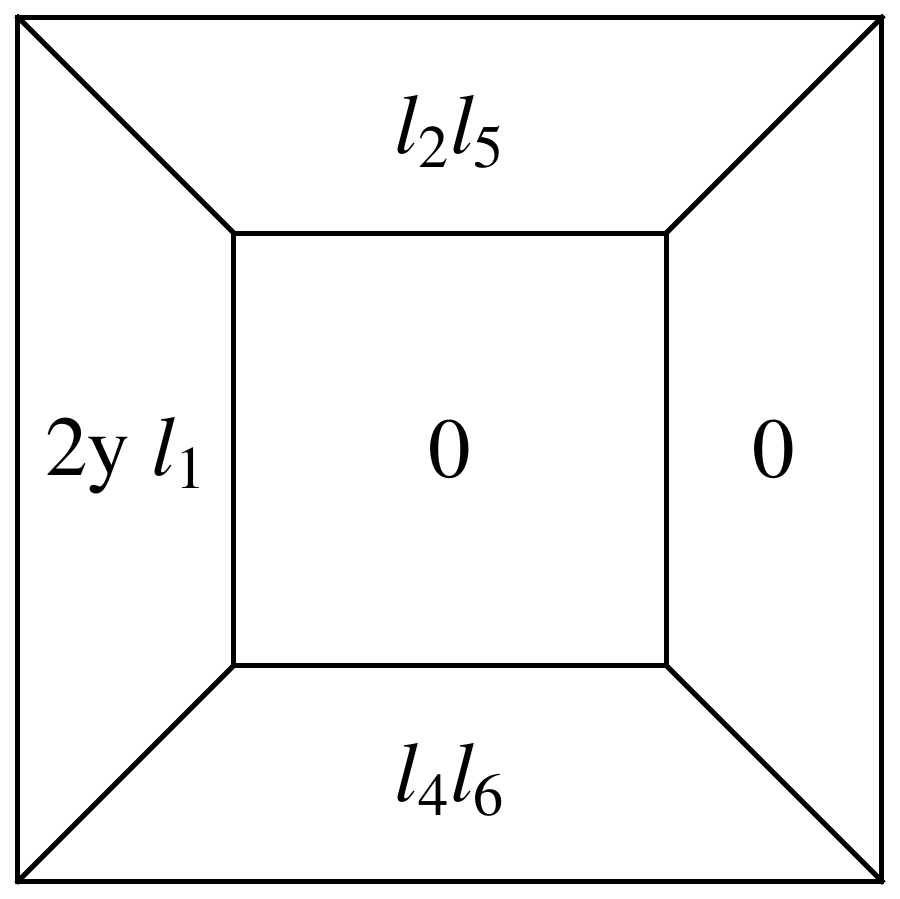}
\end{minipage}
\begin{minipage}{.02\textwidth}
\centering
$+$
\end{minipage}
\begin{minipage}{.2\textwidth}
\centering
\includegraphics[scale=.25]{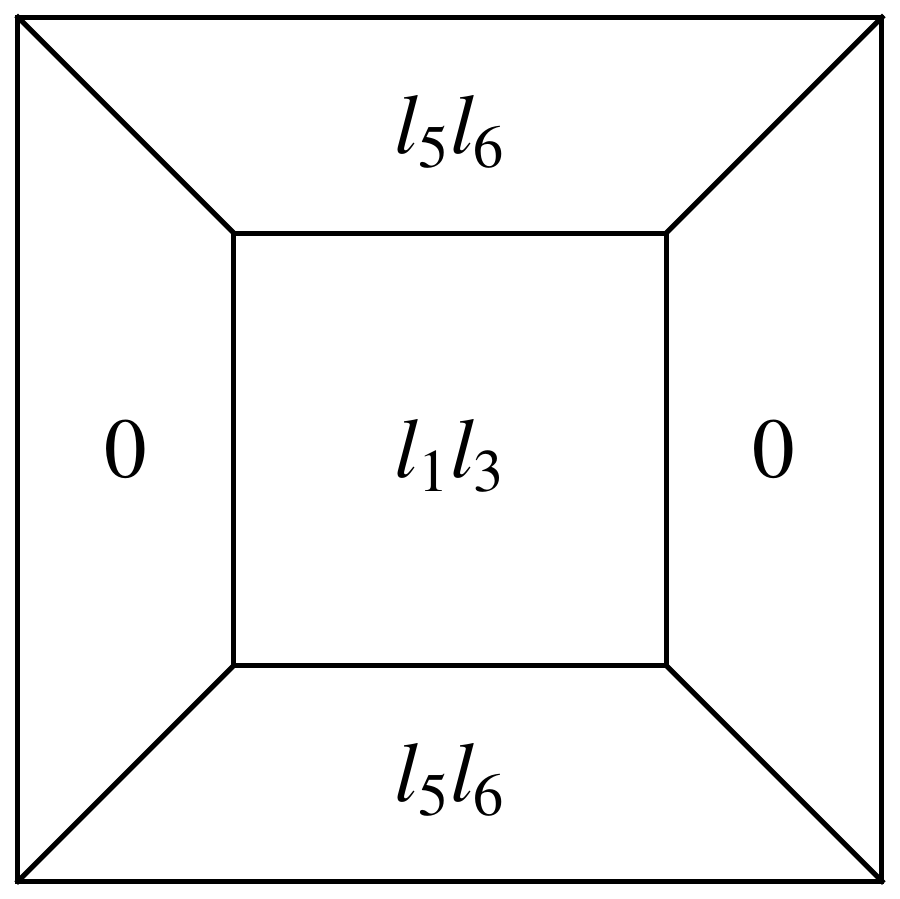}
\end{minipage}
\end{flushleft}
\caption{Decomposition of $x^2\cdot \mathbf{1}\in C^0_2(\QC)$}\label{x2decomp}
\end{figure}
\end{exm}

In \cite{HongDong},\cite{SuperSpline} and \cite{LocSup}, it is shown that in the bivariate simplicial case there exists a star-supported basis for $C^r_d(\Delta)$ if $d\geq 3r+2$ and in the trivariate simplicial case for $d>8r$.  These bases are explicitly constructed using the Bernstein-Bezier method.  Finding how large $d$ must be in order to obtain a lattice-supported basis for $C^r_d(\PC)$ is a difficult question.  The results in the bivariate and trivariate simplicial case suggest that the vector spaces $LS^{r,n}_d(\widehat{\PC})$ and $C^r_d(\widehat{\PC})$ begin to agree at a value of $d$ which is a linear function of $r$.  We address this in the planar polytopal case in $\S 5$ and relate this question to bounding the value of $d$ for which known dimension formulas for $C^r_d(\PC)$ hold.  Before progressing to this, we show how the poset $L_\PC$ can be refined to give a cleaner description of $LS^{r,k}(\PC)$.

\subsection{Refining posets for $LS^{r,k}(\PC)$}

In this section we seek a minimal set of submodules of the form $C^r_\QC(\PC)$ which generate $LS^{r,k}(\PC)$.  Observe that if $\QC\subseteq\OC$ are subcomplexes of $\PC$ then $C^r_\QC(\PC)\subseteq C^r_\OC(\PC)$.  So if there is containment among subcomplexes which are the support of the summands appearing in $LS^{r,k}(\PC)$ then we may discard one of the summands.  This suggests that while $L_\PC$ is quite useful for describing localization of $C^r(\PC)$, there is a more useful poset to consider for understanding $LS^{r,k}(\PC)$.  This is the poset $\Gamma_\PC$ whose elements are subcomplexes $\QC$ of $\PC$ which are a component of $\PC_W$ for some $W\in L_{\widehat{\PC}}$, ordered by inclusion.  As we will see, $\Gamma_\PC$ may be quite different from $L_{\widehat{\PC}}$.

Define a function $f_\Gamma$ on $\Gamma_\PC$ by 
\[
f_\Gamma(\QC)=\left\lbrace
\begin{array}{ll}
\mbox{codim}\left(\bigcap\limits_{\tau\in\QC^0_{n-1}}\mbox{aff}(\tau)\right) & \mbox{if }\QC^0_{n-1}\neq\emptyset \\
0 & \mbox{if }\QC^0_{n-1}=\emptyset
\end{array}
\right.,
\]
or equivalently 
\[
f_\Gamma(\QC)=\min\{\mbox{rk}(V)| V\in L_{\widehat{\PC}}, \QC\mbox{ a component of }\PC_V\}.
\]
$f_\Gamma$ is increasing in the sense that if $\OC\subsetneq\QC$ then $f_\Gamma(\OC)<f_\Gamma(\QC)$.  We call $f_\Gamma(\OC)$ the $\Gamma$-rank of $\OC$.  Let $\Gamma^k_\PC$ be the poset formed by $\{\QC\in\Gamma|f_\Gamma(\QC)\leq k\}$ and for a poset $L$ let $L^{\max}$ denote the maximal elements of $L$.  Then we have
\begin{prop}\label{Gamma}
$LS^{r,k}(\PC)=\sum\limits_{\QC\in\Gamma^{k,\max}_\PC} C^r_\QC(\PC)$.
\end{prop}
\begin{proof}
$C^r_W(\PC)=\sum_{i=1}^m C^r_{\PC^i_W}(\PC)$ by definition, where $\PC^1_W,\ldots,\PC^m_W$ are the components of $\PC_W$.  If $\mbox{rk}(W)\leq k$, then $f_\Gamma(\PC^i_W)\leq k$ for all components $\PC^i_W$ of $\PC_W$.  Hence $\PC^i_W\in \Gamma^k_\PC$ and $\PC^i_W\subset \QC$ for some $\QC\in\Gamma^{k,\max}_\PC$.  Since this holds for all components $\PC^i_W$ of $\PC_W$, $C^r_W(\PC)\subset \sum_{\QC\in\Gamma^{k,\max}_\PC} C^r_\QC(\PC)$ and we are done.
\end{proof}

\begin{exm}\label{Gam}
Let $\QC$ and $\QC'$ be as in Figures~\ref{SEVL} and~\ref{DSEVL}.  Label the facets of $\QC$ and $\QC'$ by $A,B,C,D,E$ as shown in Figure~\ref{GamDi}.  The Hasse diagrams of $L_{\widehat{\QC}}$ and $\Gamma_\QC$ are shown in Figure~\ref{GamDi} organized according to rank and $\Gamma$-rank, respectively.  For $L_{\widehat{\QC}}$ we use the labels assigned to the flats in Example~\ref{SchCube}.  We label the complexes in $\Gamma_\QC$ and $\Gamma_{\QC'}$ by listing the their facet labels.  Figures~\ref{cx},~\ref{xi} show the complexes $\Gamma^{2,\max}_\QC$.  Figure~\ref{dcx} shows the subcomplexes of $\Gamma^{2,\max}_{\QC'}$ which are not stars of vertices.
\begin{figure}[htp]
\captionsetup[subfigure]{labelformat=empty}
\centering
\subfloat[$\QC$]{
\includegraphics[scale=.4]{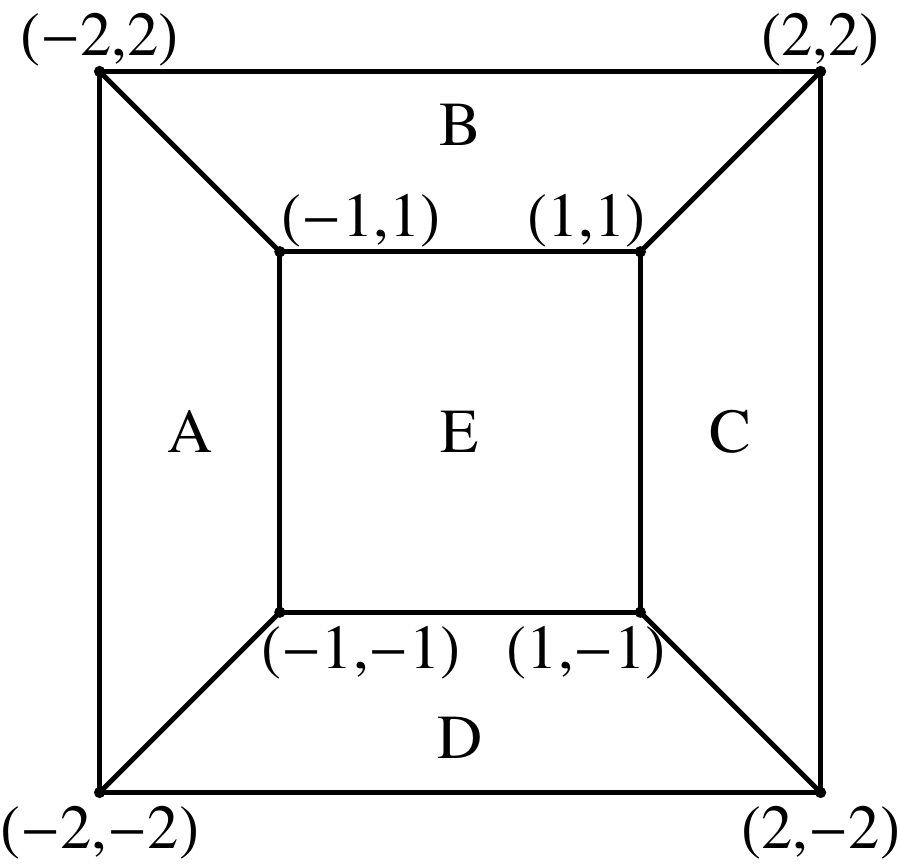}
}
\subfloat[$\QC'$]{
\includegraphics[scale=.4]{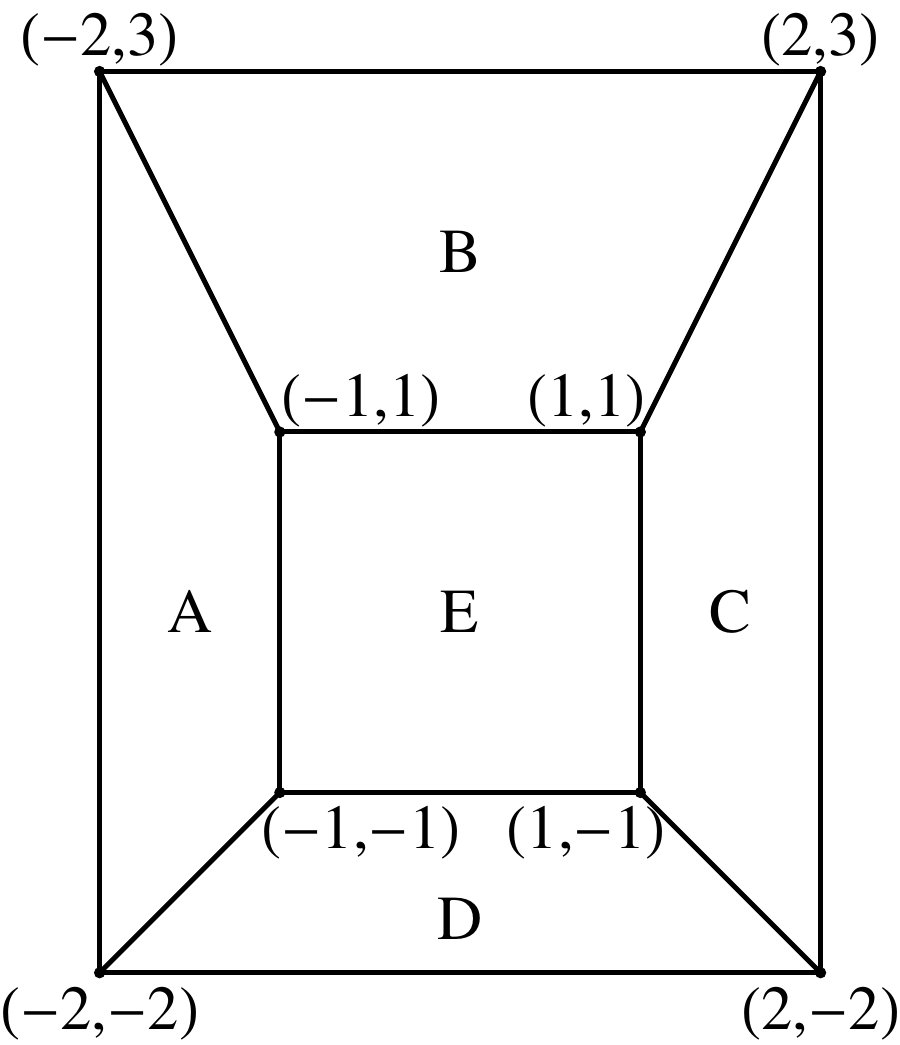}
}

\subfloat[$L_{\widehat{\QC}}$]{
\begin{tikzpicture}[scale=.6]
\node (one) at (0,-2) {$\R^2$};
\node (1) at (-5,0) {$L1$};
\node (2) at (-3,0) {$L2$};
\node (3) at (-1,0) {$L3$};
\node (4) at (1,0) {$L4$};
\node (5) at (3,0) {$L5$};
\node (6) at (5,0) {$L6$};
\node (x) at (-5,2) {$x$};
\node (y) at (-10/3,2) {$y$};
\node (z) at (-5/3,2) {$z$};
\node (w) at (0,2) {$w$};
\node (xi) at (5/3,2) {$\xi$};
\node (alph) at (10/3,2) {$\alpha$};
\node (bet) at (5,2) {$\beta$};
\draw (1) -- (one) -- (2)--(x)--(1);
\draw (3) -- (one) -- (4)--(z)--(3);
\draw (5) -- (one) -- (6)--(xi)--(5);
\draw (1)--(y)--(4);
\draw (2)--(w)--(3);
\draw (y)--(5);
\draw (w)--(5);
\draw (z)--(6);
\draw (x)--(6);
\draw (2)--(alph)--(4);
\draw (1)--(bet)--(3);
\end{tikzpicture}
}

\subfloat[$\Gamma_\QC$]{
\begin{tikzpicture}[scale=.6]
\node (A) at (-2,-2) {A};
\node (B) at (-1,-2) {B};
\node (C) at (0,-2) {C};
\node (D) at (1,-2) {D};
\node (E) at (2,-2) {E};
\node (AB) at (-5.25,0) {AB};
\node (BC) at (-3.75,0) {BC};
\node (CD) at (-2.25,0) {CD};
\node (AD) at (-.75,0) {AD};
\node (AE) at (.75,0) {AE};
\node (BE) at (2.25,0) {BE};
\node (CE) at (3.75,0) {CE};
\node (DE) at (5.25,0) {DE};
\node (ABE) at (-5.25,2) {ABE};
\node (BCE) at (-3.5,2) {BCE};
\node (CDE) at (-1.75,2) {CDE};
\node (ADE) at (0,2) {ADE};
\node (ABCD) at (1.75,2) {ABCD};
\node (AEC) at (3.5,2) {AEC};
\node (BED) at (5.25,2) {BED};
\draw (A) -- (AB) -- (B)--(BC)--(C)--(CD)--(D)--(AD)--(A)--(AE)--(E)--(BE)--(B);
\draw (C) -- (CE) -- (E);
\draw (D)-- (DE)--(E);
\draw (AB)--(ABE)--(AE)--(ADE)--(AD)--(ABCD)--(BC)--(BCE)--(BE)--(ABE);
\draw (AB)--(ABCD)--(CD)--(CDE)--(CE)--(BCE);
\draw (CDE)--(DE)--(ADE);
\draw (BE)--(BED)--(DE);
\draw (AE)--(AEC)--(CE);
\end{tikzpicture}
}

\subfloat[$\Gamma_{\QC'}$]{
\begin{tikzpicture}[scale=.6]
\node (A) at (-2,-2) {A};
\node (B) at (-1,-2) {B};
\node (C) at (0,-2) {C};
\node (D) at (1,-2) {D};
\node (E) at (2,-2) {E};
\node (AB) at (-7,0) {AB};
\node (BC) at (-5,0) {BC};
\node (CD) at (-3,0) {CD};
\node (AD) at (-1,0) {AD};
\node (AE) at (1,0) {AE};
\node (BE) at (3,0) {BE};
\node (CE) at (5,0) {CE};
\node (DE) at (7,0) {DE};
\node (ABE) at (-9,2.5) {ABE};
\node (BCE) at (-7,2.5) {BCE};
\node (CDE) at (-5,2.5) {CDE};
\node (ADE) at (-3,2.5) {ADE};
\node (ABC) at (-1,2.5) {ABC};
\node (BCD) at (1,2.5) {BCD};
\node (ACD) at (3,2.5) {ACD};
\node (ABD) at (5,2.5) {ABD};
\node (AEC) at (7,2.5) {AEC};
\node (BED) at (9,2.5) {BED};
\draw (A) -- (AB) -- (B)--(BC)--(C)--(CD)--(D)--(AD)--(A)--(AE)--(E)--(BE)--(B);
\draw (C) -- (CE) -- (E);
\draw (D)-- (DE)--(E);
\draw (AB)--(ABE)--(AE)--(ADE)--(AD)--(ABD)--(AB);
\draw (BC)--(BCE)--(BE)--(ABE);
\draw (CD)--(CDE)--(CE)--(BCE);
\draw (AD)--(ADE)--(DE)--(CDE);
\draw (BC)--(BCD)--(CD);
\draw (AB)--(ABC)--(BC);
\draw (CD)--(ACD)--(AD);
\draw (BE)--(BED)--(DE);
\draw (AE)--(AEC)--(CE);
\end{tikzpicture}
}
\caption{Example~\ref{Gam} - $\Gamma_\QC$ and $\Gamma_{\QC'}$}\label{GamDi}
\end{figure}

By Proposition~\ref{Gamma} and Theorem~\ref{main}, $C^r_k(\QC)$ and $C^r_k(\QC')$ have a basis of splines which vanish outside of the complexes $\Gamma^{\max}_\QC$ and $\Gamma^{\max}_{\QC'}$, respectively, for $k\gg 0$.  This proves the claims made in the Example~\ref{EX1}.
\end{exm}

Now suppose $\PC=\Delta$ is simplicial.  Proposition~\ref{star} shows that $\Gamma_\Delta$ is the poset of stars of faces $\tau$ so that $\mbox{aff}(\tau)$ appears in $L_\Delta$.  Every star in $\Gamma^{k,\max}_\Delta$ is contained in the star of a face $\tau\in\Delta_{n-k}$, so we obtain the following corollary to Proposition~\ref{Gamma}.
\begin{cor}\label{LSimp}
$LS^{r,k}(\Delta):=\sum_{\tau\in\Delta_{n-k}} C^r_{\mbox{st}(\tau)}(\Delta).$
\end{cor}
Setting $k=n$ and applying Theorem~\ref{main}, we obtain the existence of a star-supported basis for $C^r_d(\Delta)$ in any dimension.
\begin{cor}
Let $\Delta\subset\R^n$ be a pure, $n$-dimensional, hereditary simplicial complex.  Then $C^r_d(\Delta)$ has a basis consisting of splines supported on the star of a vertex for $d\gg 0$.
\end{cor}

\section{Lattice-Supported Splines and the McDonald-Schenck Formula}
In this section we address the question of when $\mbox{dim}_\R C^r_d(\PC)$ becomes polynomial, particularly in the planar case where these polynomials have been computed by Alfeld, Schumaker, McDonald, and Schenck (\cite{AS3r} and \cite{TSchenck08}).  Rephrased, this is a question about when the Hilbert function $HF(C^r(\widehat{\PC}),d)$ of the graded module $C^r(\widehat{\PC})$ agrees with the Hilbert polynomial $HP(C^r(\widehat{\PC}),d)$.  We give an indication of how we may use lattice-supported splines to address this problem and also give conjectural bounds on $d$ for when $HF(C^r(\widehat{\PC}),d)=HP(C^r(\widehat{\PC}),d)$ in the case $\PC\subset\R^2$.

There is a convenient notion for discussing when the Hilbert function $HF(M,d)$ of a graded module $M$ over $R=\R[x_1,\ldots,x_n]$ agrees with the Hilbert polynomial $HP(M,d)$, namely the \textit{regularity} of $M$, denoted $\mbox{reg}(M)$.  Let $F_\bullet\rightarrow M$ be a minimal free graded resolution of $M$, with $F_i=\bigoplus\limits_{j} R(-a_{ij})$.  Then
\[
\mbox{reg}(M):=\max\limits_{i,j} \{a_{ij}-i\} 
\]
The relevant facts about $\mbox{reg}(M)$ may be found in \cite{Eis}, chapter 4 and appendix A.  Regularity relates to the Hilbert function becoming polynomial via Theorem 4.2 of \cite{Eis}:
\begin{thm}\label{HFHP}
Let $M$ be a finitely generated graded module over $S=K[x_0,x_1,\ldots,x_n]$ ($K$ a field) of projective dimension $\delta$.  Then $HF(M,d)=HP(M,d)$ for $d\ge \mbox{reg}(M)+\delta-n$.
\end{thm}

We apply this theorem to $C^r(\widehat{\PC})$.  First, since $\PC$ is hereditary, $C^r(\widehat{\PC})$ is the kernel of a matrix and hence a second syzygy (see \cite{DimSeries}).  $C^r(\widehat{\PC})$ is a module over $S=\R[x_0,\ldots,x_n]$, so since $C^r(\widehat{\PC})$ is a second syzygy, $\mbox{pd}(C^r(\widehat{\PC}))\le n-1$.  By Theorem~\ref{HFHP}, $HF(C^r(\widehat{\PC}),d)=HP(C^r(\widehat{\PC}),d)$ for $d\ge \mbox{reg}(C^r(\widehat{\PC}))-1$.  It turns out that the Alfeld-Schumaker result for generic simplicial $\Delta\subset\R^2$ that $HF(C^r(\widehat{\Delta}),d)=HP(C^r(\widehat{\Delta}),d)$ for $d\ge 3r+1$ \cite{AS3r} is equivalent to $\mbox{reg}(C^r(\widehat{\Delta}))\le 3r+2$.  Schenck conjectures a tightening of this bound, namely $HF(C^r(\widehat{\Delta}),d)=HP(C^r(\widehat{\Delta}),d)$ for $d\ge 2r+1$ \cite{Thesis}.  We will call this the $2r+1$ conjecture.  This is equivalent to $\mbox{reg}(C^r(\widehat{\Delta}))\le 2r+2$ (see \cite{CohVan} for the equivalence of these statements).  We now relate $\mbox{reg}(C^r(\widehat{\PC}))$ to $\mbox{reg}(LS^{r,n}(\widehat{\PC}))$.

Theorem~\ref{latticegens} provides the short exact sequence 
\[
0\rightarrow LS^{r,n}(\widehat{\PC}) \rightarrow C^r(\widehat{\PC}) \rightarrow C \rightarrow 0,
\]
where $C$ has finite length.  The following proposition identifies $C$ as a local cohomology module of $LS^{r,n}(\widehat{\PC})$.

\begin{prop}\label{regest}
Let $\PC\subset\R^n$ be a pure hereditary polytopal complex and $C$ be the cokernel of the inclusion $LS^{r,n}(\widehat{\PC})\hookrightarrow C^r(\widehat{\PC})$.  Then
\[
C\cong H^1_m(LS^{r,n}(\widehat{\PC})),
\]
where $H^1_m(LS^{r,n}(\widehat{\PC}))$ is the first local cohomology module of $LS^{r,n}(\widehat{\PC})$ with respect to $m=(x_0,\ldots,x_n)$, the homogeneous maximal ideal of $S=\R[x_0,\ldots,x_n]$.
\end{prop}
\begin{proof}
If $M$ is a graded $S$-module, let $H^i_m(M)$ denote the $i$th local cohomology module of $M$ with respect to $m=(x_0,\ldots,x_n)$, $\widetilde{M}(i)$ the associated twisted sheaf on $\mathbb{P}^n$, and $H^0(\widetilde{M}(i))$ the vector space of global sections of $\widetilde{M}(i)$.  Define $\Gamma(M)=\bigoplus_i H^0(\widetilde{M}(i))$.  We have the four term exact sequence (see \cite{Eis} Corollary A1.12)
\[
0\rightarrow H^0_m(M) \rightarrow M \rightarrow \Gamma(M) \rightarrow H^1_m(M) \rightarrow 0.
\]
The graded modules $LS^{r,n}(\widehat{\PC})$ and $C^r(\widehat{\PC})$ determine the same sheaf since their localizations at nonmaximal primes agree by Theorem~\ref{latticegens}.  Hence $\Gamma(LS^{r,n}(\widehat{\PC}))=\Gamma(C^r(\widehat{\PC}))$.  Furthermore $C^r(\widehat{\PC})=\Gamma(C^r(\widehat{\PC}))$.  This is a consequence of the fact that $C^r(\widehat{\PC})$ is a second syzygy.  From this it follows that $Ext^i_S(C^r(\widehat{\PC}),S)=0$ for $i=n,n+1$ and hence $H^i_m(C^r(\widehat{\PC}))=0$ for $i=0,1$ by local duality (\cite{Eis} Theorem 10.6).  The four term exact sequence above then yields $C^r(\widehat{\PC})=\Gamma(C^r(\widehat{\PC}))$.

Putting this all together and using the fact that $H^0_m(LS^{r,n}(\widehat{\PC}))=0$ since $LS^{r,n}(\widehat{\PC})$ has no submodule of finite length, we arrive at the short exact sequence
\[
0\rightarrow LS^{r,n}(\widehat{\PC}) \rightarrow C^r(\widehat{\PC}) \rightarrow H^1_m(LS^{r,n}(\widehat{\PC})) \rightarrow 0
\]
So $C$, the cokernel of the inclusion $LS^{r,n}(\widehat{\PC})\hookrightarrow C^r(\widehat{\PC})$, may be identified with $H^1_m(LS^{r,n}(\widehat{\PC}))$.
\end{proof}

We record a couple of facts (see Chapter 4 or Appendix A of \cite{Eis}) about regularity and local cohomology in the following lemma.

\begin{lem}\label{regloc}
Let $M$ be a graded $S$-module, $m\subset S$ the maximal homogeneous ideal.
\begin{enumerate}
\item If $M$ has finite length, then $\mbox{reg }M=\max\limits_j\{j|M_j\neq 0\}$.
\item $H^i_m(M)$ has finite length for every $i\ge 0$.
\item $\mbox{reg}(M)=\max\limits_j \mbox{reg }H^j_m(M)+j$
\end{enumerate}
\end{lem}

\begin{cor}\label{regex2}
Let $\PC\subset\R^n$ be a pure hereditary polytopal complex.  Set $t=\mbox{reg}(LS^{r,n}(\widehat{\PC}))$.
\begin{enumerate}
\item If $d\ge t$ then $HF(LS^{r,n}(\widehat{\PC}),d)=HF(C^r(\widehat{\PC}),d)=HP(C^r(\widehat{\PC}),d)$.
\item $\mbox{reg}(C^r(\widehat{\PC}))\leq t$ and $HF(C^r(\widehat{\PC}),d)=HP(C^r(\widehat{\PC}),d)$ for $k\ge t-1$.
\end{enumerate}
\end{cor}
\begin{proof}
From Theorem~\ref{latticegens} we have the short exact sequence
\[
0\rightarrow LS^{r,n}(\widehat{\PC}) \rightarrow C^r(\widehat{\PC}) \rightarrow H^1_m(LS^{r,n}(\widehat{\PC})) \rightarrow 0,
\]
$(1)$ If $d\ge t$ then $H^1_m(LS^{r,n}(\widehat{\PC}))_d=0$ by Lemma~\ref{regloc} and $HF(LS^{r,n}(\widehat{\PC}),d)=HF(C^r(\widehat{\PC}),d)$.  We have $\mbox{pd}(LS^{r,n}(\widehat{\PC}))\le n$ since $LS^{r,n}(\widehat{\PC})$ has no submodule of finite length, so Theorem~\ref{HFHP} yields that $HF(LS^{r,n}(\widehat{\PC}),d)=HP(LS^{r,n}(\widehat{\PC}),d)$ for $d\ge t+n-n =t$.  The result follows since $HP(LS^{r,n}(\widehat{\PC}),d)=HP(C^r(\widehat{\PC}),d)$ $(2)$ If
\[
0\rightarrow A \rightarrow B \rightarrow C \rightarrow 0
\]
is a short exact sequence of graded modules, then $\mbox{reg}(B)\leq \max\{\mbox{reg}(A),\mbox{reg}(C)\}$ (see $\S 20.5$ of \cite{CommAlg}).  This fact coupled with Proposition~\ref{regest} yields $\mbox{reg}(C^r(\widehat{\PC}))\leq t$.  The second statement of $(2)$ follows from Theorem~\ref{HFHP} and the fact that $C^r(\widehat{\PC})$ is a second syzygy.
\end{proof}

In \cite{HongDong} and \cite{SuperSpline}, star-supported bases are constructed for $C^r_d(\Delta)$, $\Delta\subset\R^2$ any triangulation of a disk.  According to Proposition~\ref{regest}, this implies $H^1_m(LS^{r,2}(\widehat{\Delta}))_{3r+2}=0$, which is compatible with (but \textit{not necessarily} equivalent to) the statement $\mbox{reg } LS^{r,2}(\widehat{\Delta})\le 3r+2$.  The following conjecture is a natural generalization of this observation.

\begin{conj}\label{c1}
Let $\PC\subset\R^2$ be a hereditary polytopal complex with $F$ being the maximum length of the boundary of a polytope of $\PC$.  Then
\[
\mbox{reg}(LS^{r,2}(\widehat{\PC}))\leq F(r+1)-1.
\]
\end{conj}

It is important to note that Alfeld and Schumaker construct simplicial complexes $\Delta\subset\R^2$ so that $C^r_{3r+1}(\Delta)$ does \emph{not} have a star-supported basis.  Again via Proposition~\ref{regest}, this implies that (for these particular examples) $H^1_m(LS^{r,2}(\widehat{\Delta}))_{3r+1}\neq 0$, so $\mbox{reg } LS^{r,2}(\widehat{\Delta})\ge 3r+2$.  So if Conjecture~\ref{c1} is true, it is an optimal bound, at least in the simplicial case.  Conjecture~\ref{c1} would also imply, via Corollary~\ref{regex2}, that $\dim_\R C^r(\widehat{\PC})_k$ agrees with the McDonald-Schenck formula for $k\ge F(r+1)-2$.  We also propose the following generalization of Schenck's $2r+1$ conjecture.

\begin{conj}\label{c2}
Let $\PC\subset\R^2$ be a hereditary polytopal complex with $F$ being the maximum length of the boundary of a polytope of $\PC$.  Then
\[
\mbox{reg}(C^r(\widehat{\PC}))\leq (F-1)(r+1).
\]
\end{conj}

This conjecture would imply, via Corollary~\ref{regex2}, that $\dim_\R C^r(\widehat{\PC})_k$ agrees with the McDonald-Schenck formula for $k\ge (F-1)(r+1)-1$.  The simplicial case of this, Schenck's $2r+1$ conjecture, is still open.  See \cite{CohVan} for an approach using the cohomology of sheaves on $\mathbb{P}^2$ and \cite{Stefan} for an example showing that this bound is tight.  We give a family of examples which show that the regularity bound of Conjecture~\ref{c2} cannot be lowered further.

\begin{thm}
There exists a polytopal complex $\PC\subset\R^2$ having one triangular face, $n-1$ quadrilateral faces, and two $(n+1)$-gons, such
that $C^r(\widehat{\PC})$ has a minimal generator of degree $n(r+1)$ supported on a single facet.  Hence $\mbox{reg}(C^r(\widehat{\PC}))\ge n(r+1)$ and Conjecture~\ref{c2} cannot be made tighter.
\end{thm}
\begin{proof}
Let $T_n\subset\R^2$ be the polyhedral complex with
\begin{itemize}
\item $2n+1$ vertices as follows: $v_0=w_0=(0,0)$, $v_i=(i,i(i+1)/2)$ for $i=1,\ldots,n$, and $w_j=(j,j(j+1)/2)$ for $j=1,\ldots,n$
\item $1$ triangular face $P_0$ with vertices $(0,0),v_1,w_1$
\item $n-1$ quadrilateral faces $P_i$ with vertices $v_i,w_i,v_{i+1},w_{i+1}$ for $i=1,\ldots,n-1$
\item Two $(n+1)$-gons $A$ and $B$ with vertices $(0,0),v_1,\ldots,v_n$ and $(0,0),w_1,\ldots,w_n$, respectively. (See Figure~\ref{T3})
\end{itemize}

\begin{figure}[htb]
\centering
\includegraphics[scale=.6]{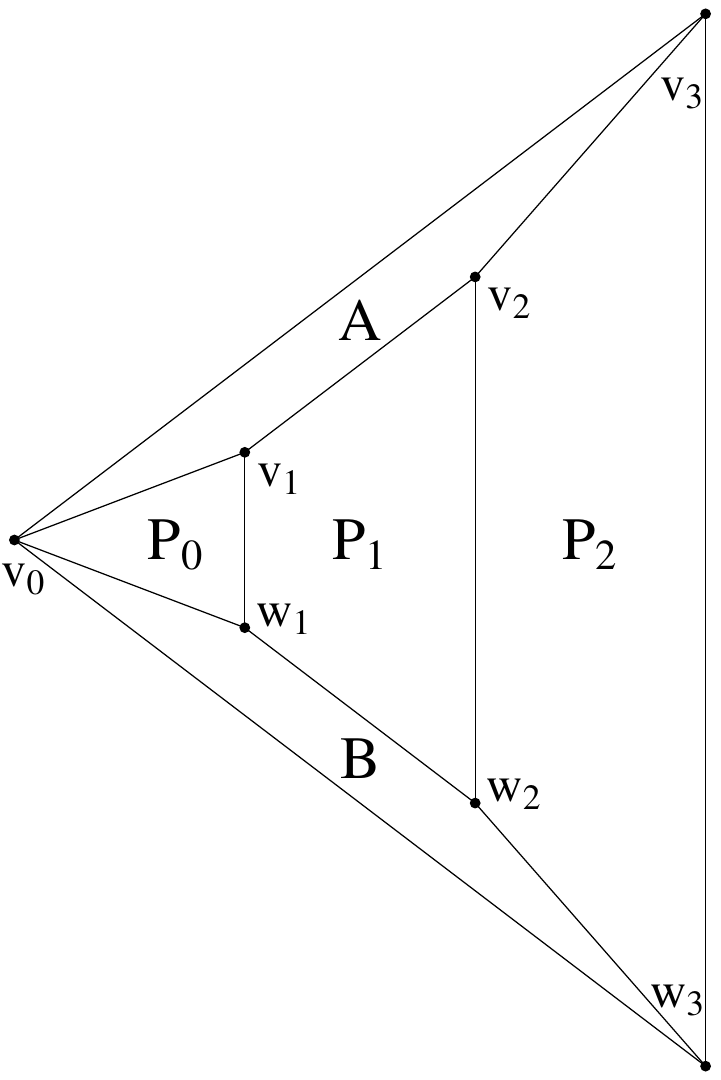}
\caption{$T_3$}\label{T3}
\end{figure}

Set $S=\R[x,y,z]$,$R=\R[x,z]$.  $u_k=(k+1)x-y-\binom{k+1}{2}z, h_k=x-kz,l_k=(k+1)x+y-\binom{k+1}{2}z$ are the homogenized forms defining the edges between $v_k$ and $v_{k+1}$, $v_k$ and $w_k$, $w_k$ and $w_{k+1}$, respectively.  Let $\phi_A:C^r(\widehat{T}_n)\rightarrow S$ denote the map obtained by restricting splines to the facet $A$ and set $NT^r_n=\mbox{ker }\phi_A$.

Suppose we have $G\in NT^r_n$.  Then $u^{r+1}_i|G_{\widehat{P}_i}$ and $l^{r+1}_i|G_{\widehat{B}}-G_{\widehat{P}_i}$ for $i=0,\ldots,n-1$.  So $G_{\widehat{B}}\in(u^{r+1}_i,l^{r+1}_i)$ for $i=0,\ldots,n-1$ and $G_{\widehat{B}}\in \cap_{i=0}^{n-1} (u^{r+1}_i,l^{r+1}_i)=J_r$.  Define $p_y:S\rightarrow R$ by $p_y(f(x,y,z))=f(x,0,z)$ for $f(x,y,z)\in S$.  $p_y(u_k)=p_y(l_k)=(k+1)x-\binom{k+1}{2}$, so $p_y(J_r)=I_r\subset R$ is the principal ideal $\cap_{i=0}^{n-1} (x-(k/2)z)^{r+1}=(\prod_{k=0}^{n-1} (x-(k/2)z))^{r+1}$.

So we have a graded homomorphism of $S$-modules $p_y\circ\phi_B:NT^{r+1}_n \rightarrow I_r$, where $\phi_B(G)=G_{\widehat{B}}$.  Define $F(B)\in C^r_{\widehat{B}}(\widehat{T}_n)$ by $F(B)_{\widehat{\sigma}}=0$ for every facet $\sigma$ of $T_n$ other than $B$ and $F(B)_{\widehat{B}}=L_{\partial \widehat{B}}$.  $(p_y\circ\phi_B)(F(B))=(\prod_{k=0}^{n-1} (k+1)(x-(k/2)z))^{r+1}$, which is a minimal generator of the ideal $I_r$.  It follows that $F(B)$ is a minimal generator of $C^r(\widehat{T_n})$.  Since $F(B)$ has degree $n(r+1)$, we are done.
\end{proof}

The following corollary indicates how different the polytopal case is from the simplicial, in which case $C^r(\widehat{\Delta})$ is generated as a module over $S=\R[x,y,z]$ in degree at most $3r+2$.

\begin{cor}
If $\PC$ is planar polyhedral complex, then $C^r(\widehat{\PC})$ may be generated as an $S$-module in arbitrarily high degree.
\end{cor}

\textbf{Acknowledgements:}  I would like to thank my advisor, Hal Schenck, for suggesting this direction of research, and Alexandra Seceleanu for careful reading and great suggestions.

\end{document}